\documentclass[11pt]{amsart}
\DeclareRobustCommand{\VAN}[3]{#2}
\usepackage[style=numeric-comp,sorting=nyt]{biblatex}

\usepackage[utf8]{inputenc}
\usepackage{tikz-cd}
\tikzstyle{printersafe}=[snake=snake,segment amplitude=0 pt]
\usepackage[font=small,format=plain,labelfont=bf,up,textfont=it,up]{caption}
\usepackage[margin=1in]{geometry} 
\usepackage{fancyhdr}
\usepackage{hyperref}
\usepackage{stmaryrd}
\usepackage{amsmath}
\usepackage{mathtools} 
\usepackage{thm-restate}
\usepackage{mathrsfs}
\usepackage{amsfonts}
\usepackage{amssymb} 
\usepackage{amsthm}
\usepackage{stmaryrd}
\usepackage[OT2,T1]{fontenc}
\DeclareSymbolFont{cyrletters}{OT2}{wncyr}{m}{n}
\DeclareMathSymbol{\Sha}{\mathalpha}{cyrletters}{"58}
\newtheorem{Thm}[subsubsection]{Theorem}
\newtheorem{Lem}[subsubsection]{Lemma}
\newtheorem{Prop}[subsubsection]{Proposition}
\newtheorem{Cor}[subsubsection]{Corollary}

\newtheorem{Conj}[subsubsection]{Conjecture}
\newtheorem*{Thm*}{Theorem}
\newtheorem{mainThm}{Theorem}

\theoremstyle{definition}
\newtheorem{Def}[subsubsection]{Definition}

\newtheorem{Question}[subsubsection]{Question}

\newtheorem{Rem}[subsubsection]{Remark}
\newtheorem{Assump}[subsubsection]{Assumption}
\numberwithin{equation}{subsection}

\newcommand{\spec}{\operatorname{Spec}}

\newcommand{\gal}{\operatorname{Gal}}

\newcommand\restr[2]{{% we make the whole thing an ordinary symbol
  \left.\kern-\nulldelimiterspace % automatically resize the bar with \right
  #1 % the function
  \vphantom{\big|} % pretend it's a little taller at normal size
  \right|_{#2} % this is the delimiter
  }}

\newcommand{\shginf}{\operatorname{Sh}_{G}}
\newcommand{\shginfl}{\operatorname{Sh}_{G,K^{p,\ell}}}
\newcommand{\shg}{\operatorname{Sh}_{G, K^p}}
\newcommand{\shgsp}{\operatorname{Sh}_{\mathcal{G}_{V},\mathcal{K}^p}}
\newcommand{\shgspinf}{\operatorname{Sh}_{\mathcal{G}_{V}}}

\newcommand{\shgb}{\operatorname{Sh}_{G,[b],K^p}}
\newcommand{\shgbinf}{\operatorname{Sh}_{G,[b]}}

\newcommand{\afp}{\mathbb{A}_{f}^{p}}
\newcommand{\af}{\mathbb{A}_{f}}
\newcommand{\afs}{\mathbb{A}_{f}^{\Sigma}}
\newcommand{\ovfp}{\overline{\mathbb{F}}_{p}}

\newcommand{\gder}{G^{\mathrm{der}}}
\newcommand{\gad}{G^{\mathrm{ad}}}
\newcommand{\gab}{G^{\mathrm{ab}}}

\newcommand{\ql}{\mathbb{Q}_{\ell}}
  
\newcommand{\qpbr}{\breve{\mathbb{Q}}_{p}}
\newcommand{\zpbr}{\breve{\mathbb{Z}}_{p}}
\newcommand{\qp}{\mathbb{Q}_p}
\newcommand{\zp}{\mathbb{Z}_p}

\newcommand{\qlbar}{\overline{\mathbb{Q}}_{\ell}}
\newcommand{\qpbar}{\overline{\mathbb{Q}}_{p}}
\newcommand{\iggsp}{\operatorname{Ig}_{[b],\mathcal{G}_V,\mathcal{K}^p}}
\newcommand{\igmgsp}{\operatorname{Ig}_{M,\llbracket b \rrbracket,\mathcal{G}_V,\mathcal{K}^p}}
\newcommand{\ig}{\operatorname{Ig}_{[b],K^p}}
\newcommand{\igm}{\operatorname{Ig}_{M,\llbracket b \rrbracket,K^p}}
\newcommand{\iginf}{\operatorname{Ig}_{[b]}}
\newcommand{\xmub}{X_{\{\nu\}}(b)}

\newcommand{\admu}{\operatorname{Adm}(\{\nu\})}
\newcommand{\cy}{C_{y,K^p}}
\newcommand{\cby}{C_{\llbracket b_y\rrbracket,K^p}}
\newcommand{\cbynaught}{C_{\llbracket b_{1}\rrbracket,K^p}}
\newcommand{\cb}{C_{\llbracket b\rrbracket,K^p}}
\newcommand{\cbinf}{C_{\llbracket b\rrbracket}}

\newcommand{\cbcirc}{C_{\llbracket b \rrbracket,K^p}^{\circ}}

\newcommand{\cbyinf}{C_{\llbracket b_y \rrbracket}}
\newcommand{\cbgsp}{C_{\llbracket b \rrbracket,\mathcal{G}_{V},\mathcal{K}^p}}

\newcommand{\jbder}{J_{b}^{\mathrm{der}}}
\newcommand{\mon}{\operatorname{Mon}}
\newcommand{\mbg}{\mathbf{G}}
\newcommand{\mbgder}{\mathbf{G}^{\mathrm{der}}}
\newcommand{\mbgab}{\mathbf{G}^{\mathrm{ab}}}

\newcommand{\mbm}{\mathbf{M}}
\newcommand{\mbmder}{\mathbf{M}^{\mathrm{der}}}
\newcommand{\mbmab}{\mathbf{M}^{\mathrm{ab}}}
\newcommand{\jbderan}{J_{b, \mathrm{an}}^{\mathrm{der}}}
\newcommand{\jbderiso}{J_{b, \mathrm{iso}}^{\mathrm{der}} }

\usepackage{float}
\mathtoolsset{showonlyrefs}
\usepackage{bbm}
\usepackage{graphicx}
\usepackage{enumitem}
\usepackage{hyperref}
\usetikzlibrary{decorations.pathmorphing}
\tikzstyle{printersafe}=[decoration={snake,amplitude=0pt}]

\author{Pol van Hoften} 
\address{Department of Mathematics, Vrije Universiteit Amsterdam, De Boelelaan 1111, 1081 HV Amsterdam, The Netherlands}
\email{p.van.hoften@vu.nl}
\author{Luciena Xiao Xiao}
\address[Luciena Xiao Xiao]{Department of Mathematics and Statistics, University of Helsinki, Pietari Kalmin katu 5, 00560 Helsinki, Finland}
\email{xiao.xiao@helsinki.fi}
\usepackage{xcolor}

\thanks{Pol van Hoften was supported by the Engineering and Physical Sciences Research Council [EP/L015234/1], The EPSRC Centre for Doctoral Training in Geometry and Number Theory (The London School of Geometry and Number Theory), University College London and King's College London. L.X. Xiao was supported by ERC Advanced grant 742608 ``GeoLocLang'', UMR 7586 IMJ-PRG and CNRS}

\title{Monodromy and irreducibility of Igusa varieties}
\begin{document}
\begin{abstract}
We determine the irreducible components of Igusa varieties for Shimura varieties of Hodge type under a mild condition and use that to compute the irreducible components of central leaves. In particular, we show that a strong version of the discrete Hecke orbit conjecture is false in general. Our method combines recent work of D'Addezio on monodromy groups of compatible local systems with a generalisation of a method of Hida, using the Honda--Tate theory for Shimura varieties of Hodge type developed by Kisin--Madapusi Pera--Shin. We also determine the irreducible components of Newton strata in Shimura varieties of Hodge type by combining our methods with recent work of Zhou--Zhu.
\end{abstract}
\maketitle
%\tableofcontents
\section{Introduction}
Let $N \ge 4$, let $p$ be a prime number coprime to $N$, let $Y_1(N)$ be the modular curve of level $\Gamma_1(N)$ over $\mathbb{F}_p$ and let $Y_1(N)^{\text{ord}}$ be the ordinary locus. There is a tower of finite étale covers (see \cites{Igusa}) $\operatorname{Ig}_{m} \to Y_1(N)^{\text{ord}}$ with Galois groups $(\mathbb{Z}/p^m \mathbb{Z})^{\times}$, and we let $\operatorname{Ig}_{\infty} \to Y_1(N)^{\text{ord}}$ be their inverse limit. It is a classical result due to Igusa that $\operatorname{Ig}_{\infty}$ is irreducible.

Igusa varieties exist more generally as profinite étale covers of central leaves in the special fibers of Shimura varieties of Hodge type (cf. \cite{HamacherKim}*{Section 5}; see \cites{MantovanPEL} for the PEL case). Understanding the irreducible components of these Igusa varieties has important consequences for the theory of $p$-adic automorphic forms. For example, in the work of Eischen--Mantovan \cite{eischen2017p} on $p$-adic automorphic forms for unitary Shimura varieties, the irreducibility of Igusa varieties is assumed throughout.

\subsection{Main results}
Let $(G,X)$ be a Shimura datum of Hodge type with reflex field $E$. Let $p>2$ be a prime number, let $K^p \subset G(\afp)$ be a sufficiently small compact open subgroup and let $K_p \subset G(\qp)$ be a hyperspecial subgroup. For a prime $v | p$ of $E$, we let $\operatorname{Sh}_G$ be the base change to $\ovfp$ of the canonical integral model over $\mathcal{O}_{E, v}$ of the Shimura variety of level $K^p K_p$, see \cite[main theorem]{KisinModels}.

Let $\operatorname{Sh}_{G,[b]} \subset \operatorname{Sh}_G$ be a non-basic Newton stratum and let $C \subset \operatorname{Sh}_{G,[b]}$ be a central leaf (see \cite{HamacherKim} or Section \ref{Sec:NewtonStratification}). Then the Igusa variety $\operatorname{Ig}_{[b]} \to C$ is a profinite \'etale cover with Galois group a compact open subgroup $H_C \subset J_b(\qp)$, where $J_b(\qp)$ is the twisted centraliser in $G(\qpbr)$ of some $b \in [b]$ (see Section \ref{Sec:TwistedCentralisers}). Let $\gder$ denote the derived group of $G$. Write $G^{\mathrm{ab}}$ for the maximal abelian quotient of $G$ and $G^{\mathrm{ab}}(\zp)$ for the $\zp$-points of the connected N\'eron model of $G^{\mathrm{ab}}_{\qp}$. Under the natural surjective map $J_b(\qp) \to G^{\mathrm{ab}}(\qp)$ (see Section \ref{Sec:TwistedCentralisers}) the subgroup $H_C$ maps to $G^{\mathrm{ab}}(\zp)$, which defines an action of $H_C$ on $\gab(\zp)$. %Let $h_G$ be the Coxeter number of $G$.
\begin{mainThm} \label{Thm:Monodromy}
Assume that $G^{\mathrm{der}}$ is simply connected and $\mathbb{Q}$-simple. If $J_b^{\mathrm{der}}$ has no compact factors, then the natural map (induced by $\operatorname{Ig}_{[b]} \to C \to \operatorname{Sh}_G$)
\begin{align}
    \pi_0(\operatorname{Ig}_{[b]}) \to \pi_0(\operatorname{Sh}_G),
\end{align}
is surjective with fibers in bijection with $G^{\mathrm{ab}}(\zp)$, equivariant for the action of $H_C$.
\end{mainThm}
The natural map in Theorem \ref{Thm:Monodromy} is equivariant for the prime-to-$p$ Hecke operators, but it should not be true that these operators act trivially on the fibers, see Conjectures \ref{Conj:Igusa} and \ref{Conj:DiscreteHOWeak}.
\begin{Rem}
In the case of the modular curve, the ordinary Igusa variety $\operatorname{Ig}_{[\mathrm{ord}]}$ is a $(\mathbb{Z}_p^{\times})^2$-torsor over the ordinary locus, and our theorem tells us that its connected components are in bijection with $\mathbb{Z}_p^{\times}$; here $(\mathbb{Z}_p^{\times})^2$ acts on $\mathbb{Z}_p^{\times}$ via the product map. This recovers the result of Igusa from the first paragraph of the introduction, because the Igusa tower $\operatorname{Ig}_{\infty}$ introduced there is the inverse image of $1$ under $\operatorname{Ig}_{[\mathrm{ord}]} \to \mathbb{Z}_p^{\times}$.
\end{Rem}
The irreducibility of Igusa varieties was proved for Siegel modular varieties by Chai-Oort \cites{ChaiOort}, and their proof works more generally for Shimura varieties of PEL type when hypersymmetric points exist (cf. \cites{hida2011irreducibility,eischen2017p}). We would like to point out that even in the $\mu$-ordinary locus, hypersymmetric points often do not exist (see \cite{LXXiao}*{Corollary 7.5.}). Hida, see \cite{Hida}, proved the irreducibility of the ordinary Igusa tower over Shimura varieties of PEL type A and C without using hypersymmetric points. Our results are the first to treat Hodge type Shimura varieties and Igusa varieties over general central leaves (but see Remark \ref{Rem:KretShin}); they are even new for the $\mu$-ordinary locus in many PEL type cases. 

When $[b]$ is basic, the Igusa variety $\operatorname{Ig}_{[b]}$ is zero-dimensional and highly reducible. In particular, the theorem is false for products of Shimura varieties with $[b]$ basic in one factor and non-basic in the other; this is where the assumption that $G^{\mathrm{der}}$ is $\mathbb{Q}$-simple comes from. This assumption is equivalent to asking that the adjoint group $\gad$ is $\mathbb{Q}$-simple, since we also assume that $\gder$ is simply connected. It can be replaced with the assumption that $[b]$ is $\mathbb{Q}$-non-basic, see Section \ref{Sec:Monodromy}.
\begin{Rem}
The assumption that $J_b^{\mathrm{der}}$ has no compact factors is relatively mild; For Siegel modular varieties, it comes down to the assumption that the $F$-isocrystal corresponding to $[b]$ does not have slope $1/2$ with multiplicity one. It is automatic for the $\mu$-ordinary locus or more generally for Newton strata corresponding to $[b]$ with $J_b$ quasi-split.
\end{Rem}
We can use Theorem \ref{Thm:Monodromy} to determine the irreducible (equivalently, connected) components of the central leaf $C$.
\begin{Cor} \label{Cor:LeavesIrred}
Assume that $G^{\mathrm{der}}$ is simply connected and $\mathbb{Q}$-simple. If $J_b^{\mathrm{der}}$ has no compact factors, then the natural map
\begin{align}
    \pi_0(C) \to \pi_0(\operatorname{Sh}_G)
\end{align}
is surjective with finite fibers given by the quotient $H_C \backslash G^{\mathrm{ab}}(\zp)$.
\end{Cor}
A strong version of the discrete Hecke orbit conjecture, see \cite{KretShin}*{Question 8.2.1.$\operatorname{HO}^+_{\mathrm{disc}}$},
predicts that the natural map $\pi_0(C) \to \pi_0(\operatorname{Sh}_G)$ is an isomorphism. Using Corollary \ref{Cor:LeavesIrred}, the conjecture comes down to the surjectivity of $H_C \to G^{\mathrm{ab}}(\zp)$. We will show that if $H_C$ is a parahoric subgroup, then $H_C$ surjects onto $G^{\mathrm{ab}}(\mathbb{Z}_p)$; this is also used in our proof of Theorem \ref{Thm:Monodromy}. In Section \ref{counterexample}, we give an example, due to Rong Zhou, which shows that this equality does not always hold. In particular, \cite{KretShin}*{Question 8.2.1.$\operatorname{HO}^+_{\mathrm{disc}}$} has a negative answer in general. \medskip 

Our second main result is about irreducible components of Newton strata.
\begin{mainThm} \label{Thm:Newton}
Assume that $G^{\mathrm{der}}$ is simply connected and $\mathbb{Q}$-simple. If $J_b^{\mathrm{der}}$ has no compact factors, then the natural map
\begin{align}
    \pi_0(\operatorname{Sh}_{G,[b]}) \to \pi_0(\operatorname{Sh}_G)
\end{align}
is a bijection. Moreover, the number of irreducible components in each connected component of $\operatorname{Sh}_{G,[b]}$ is given by the representation-theoretic constant
\begin{align}
    \operatorname{Dim} V_{\mu}(\lambda_b)_{\mathrm{rel}},
\end{align}
introduced in \cite[Section 2.6]{ZhouZhu}. In particular, if $G_{\qp}$ is split, then the connected components of $\operatorname{Sh}_{G,[b]}$ are irreducible.
\end{mainThm}
Theorems \ref{Thm:Monodromy} and \ref{Thm:Newton} were proved for Siegel modular varieties by Chai and Oort in their seminal paper \cites{ChaiOort}. Amusingly, they prove irreducibility of Newton strata first, irreducibility of central leaves second and irreducibility of Igusa varieties last. \smallskip

\begin{Rem}
In Section \ref{Sec:MainThm} we prove versions of our main theorems beyond the case that $K_p$ is hyperspecial. To be precise, we work with the Igusa varieties constructed by Hamacher--Kim, see \cite{HamacherKim}, over the Kisin--Pappas integral models of Shimura varieties of Hodge type, see \cite{KisinPappas}, when $K_p$ is a connected very special parahoric. See Theorems \ref{Thm:Monodromy2} and \ref{Thm:Newton2} for the more general versions of Theorems \ref{Thm:Monodromy} and \ref{Thm:Newton} and Corollary \ref{Cor:LeavesIrred2} for the general version of Corollary \ref{Cor:LeavesIrred}.
\end{Rem}
\begin{Rem}\label{Rem:KretShin} In recent work \cites{KretShin}, Kret and Shin also determine the irreducible components of Igusa varieties when $G_{\qp}$ is unramified, and they moreover prove the discrete Hecke orbit conjecture (Conjecture \ref{Conj:DiscreteHOWeak}). Their proof uses harmonic analysis and automorphic forms and is completely different from ours. They compute the $0$-th étale cohomology of $\operatorname{Ig}_{[b]}$ as a representation of $G(\afp) \times J_b(\qp)$ using the Langlands--Kottwitz method (\cite{KretShin}*{Theorem A}), and then determine the irreducible components of $\operatorname{Ig}_{[b]}$ using that computation. It might be possible to recover their computation of the $0$-th étale cohomology from Theorem \ref{Thm:Monodromy}, see Conjecture \ref{Conj:Igusa} and Question \ref{Q:ToriGeneratedGab}.
\end{Rem}

\subsection{Strategy}
\subsubsection{Setup}
Recall that the Igusa variety is a profinite étale cover $\operatorname{Ig}_{[b]} \to C$ of a central leaf $C$ inside the Newton stratum $\operatorname{Sh}_{G,[b]}$. To be precise, it is a torsor for a compact open subgroup $H_C \subset J_b(\qp)$, where $J_b$ is the twisted centraliser in $G(\qpbr)$ attached to some $b \in [b]$ (see Section \ref{Sec:Notation}). There are many different central leaves $C$ inside the Newton stratum $\operatorname{Sh}_{G,[b]}$, all giving rise to isomorphic Igusa varieties, but the group $H_C$ \emph{does} depend on $C$. For the purposes of our proof we will always choose $C$ to be a \emph{distinguished} central leaf, that is, a central leaf that is also an Ekedahl--Oort stratum; these always exist by \cite[Theorem D]{ShenZhang}, see Section \ref{Sec:Distinguished}. Informally, distinguished central leaves are the `smallest possible' central leaves $C$ and correspondingly $H_C$ is `as large as possible' when $C$ is distinguished. For the rest of this section, we assume that $C$ is distinguished and write $H$ for $H_C$. \smallskip 

The algebraic group $J_b$ is an inner form of a Levi subgroup of $G_{\qp}$, and therefore has a surjective map $J_b \to G^{\mathrm{ab}}$. Let $J_b'$ be the kernel of this map and let $H'=H \cap J_b'(\qp)$. In Section \ref{Sec:Parahoric}, we will show that $H/H'$ is isomorphic to $G^{\mathrm{ab}}(\mathbb{Z}_p)$. When $C$ is not distinguished, the quotient $H_C/H_C'$ is in general strictly contained in $\gab(\zp)$, see Section \ref{counterexample}.

\subsubsection{Outline of the proof} Theorem \ref{Thm:Newton} is a consequence of Theorem \ref{Thm:Monodromy} together with a careful analysis of the Mantovan product-formula, due to Hamacher--Kim in this generality, and results of Chen--Kisin--Viehmann, Zhou--Zhu and He--Zhou, see \cites{MantovanPEL, HamacherKim,ChenKisinViehmann, ZhouZhu,He-Zhou}. Below, we will outline the proof of Theorem \ref{Thm:Monodromy}. \medskip 

Under the assumptions of our main theorem, distinguished central leaves are irreducible. To be precise, \cite{vH}*{Theorem 4} tells us that for a distinguished $C$ the natural map
\begin{align} \label{Eq:FibersEquation}
    \pi_0(C) \to \pi_0(\operatorname{Sh}_G)
\end{align}
is a bijection. Therefore, the main theorem would follow if we could show that the fibers of $\pi_0(\ig) \to \pi_0(C)$ are in bijection with $G^{\mathrm{ab}}(\mathbb{Z}_p)$, equivariant for the action of $H$ via $H \to \gab(\zp)$. This comes down to showing that the stabiliser in $H$ of a connected component of $\ig$ is given by $H'$. \smallskip 

Fix a connected component $C^{\circ}$ of $C$. The $H$-torsor $\operatorname{Ig}_{[b]}^{\circ} \to C^{\circ}$ corresponds to a continuous morphism
\begin{align}
    \rho_{\operatorname{Ig}}: \pi_1^{\text{ét}}(C^{\circ}) \to H,
\end{align}
with image $M$. If we could show that $M=H'$, then the fibers of $\operatorname{Ig}_{[b]}^{\circ} \to C^{\circ}$ are given by $H/H'=\gab(\zp)$. Let $\mathbb{M}$ be the Zariski closure of $M$ inside $J_b(\qp)$, which is an algebraic group over $\qp$.

Let $A \to C^{\circ}$ be the `universal' abelian variety coming from a choice of Hodge embedding. It is a consequence of \cite{vH}*{Theorem 4} that the $\ell$-adic monodromy of $A$ over $C^{\circ}$ is maximal. This can be combined with two results of D'Addezio from \cites{d2020monodromy, DAddezioII} to show that a certain (overconvergent) $p$-adic monodromy group over $C^{\circ}$ is maximal. We deduce from this that $\mathbb{M}= J_b'$ using Proposition \ref{Prop:FactIsocrystals}, which
compares $\mathbb{M}$ with the monodromy group of the $F$-isocrystal associated to $A$ over $C^{\circ}$; this proposition might be of independent interest. As a corollary, we find that $M$ contains a compact open subgroup of $\jbder(\qp)$ and that $M$ is contained in $J_b'(\qp)$. In order to show that $M=H'$, we will make use of the fact that the action of $H$ on $\operatorname{Ig}_{[b]}$ extends to an action of $J_b(\qp)$. The goal is to show that the action of $J_b'(\qp)$ on $\pi_0(\operatorname{Ig}_{[b]})$ is trivial, which then implies that $M=H \cap J_b'(\qp)=H'$. 

\subsubsection{} We will show in Section \ref{Sec:Monodromy} via a group-theoretic argument that $\jbder(\qp)$ acts trivially on $\pi_0(\operatorname{Ig}_{[b]})$. The main ingredients are the equality $\mathbb{M}= J_b'$ mentioned above and the fact that the $\qp$-points of semisimple and simply connected groups with no compact factors have no finite index proper subgroups. We then show in Section \ref{Section4} that for each connected component $Z$ of $\operatorname{Ig}_{[b]}$, there is a maximal torus $T' \subset J_b'$ such that $T'(\qp)$ stabilises $Z$. In fact, we can show that this is true for all maximal tori $T'$ in $J_b'$, up to isomorphism of tori. This is done by generalising an argument of Hida \cite{Hida} using the Honda--Tate theory for Shimura varieties of \cite{KMPS}. \smallskip 

In Section \ref{Sec:GroupTheory}, we will deduce from this that $J_b'(\qp)$ acts trivially on $\pi_0(\operatorname{Ig}_{[b]})$. Indeed, we will show that $J_b'(\qp)$ is (topologically) generated by $\jbder(\qp)$ and $T_1(\qp), \cdots, T_n(\qp)$, where $T_i$ are maximal tori of $J_b'$, which may be specified up to isomorphism, see Proposition \ref{Prop:GroupTheory}. When $J_b'$ is quasi-split one only needs a single torus $T_1$, namely the centraliser of a maximal split torus. When $J_b'$ is not quasi-split, one needs $2^m$ tori for some explicit $m$ depending on $[b]$, see the statement of Proposition \ref{Prop:GroupTheory}.

In Section \ref{Sec:MainThm} we state and prove the general versions of our main results.

\section{Shimura varieties and Igusa varieties} \label{Sec:Notation}
The goal of this section is to recall the integral models of Shimura varieties of Hodge type constructed in \cites{KMPS,KisinPappas}, and the constructions of central leaves and Igusa varieties from \cite{HamacherKim}.

\subsubsection{Hodge cocharacters} \label{Sec:HodgeCocharacters} If $(G,X)$ is a Shimura datum, then for each $x \in X$ there is a cocharacter $\mu_x:\mathbb{G}_{m,\mathbb{C}} \to G_{\mathbb{C}}$, see \cite[Section 1.2.3]{KMPS} for the precise definition. The $G(\mathbb{C})$-conjugacy class of $\mu_x$ does not depend on the choice of $x$ and we will denote it by $\{\mu\}$. This conjugacy class of cocharacters is defined over a number field $E \subset \mathbb{C}$, called the reflex field. Given a place $v$ of $E$ above a rational prime $p$ and a choice of algebraic closure $E_v \to \qpbar$, there is an induced $G(\qpbar)$-conjugacy class of cocharacters of $G_{\qp}$, which we will also denote by $\{\mu\}$. 
\subsection{The construction of integral models} \label{Sec:Models}
For a symplectic space $(V, \psi)$ over $\mathbb{Q}$, we write $\mathcal{G}_V:=\operatorname{GSp}(V, \psi)$ for the group of symplectic similitudes of $V$ over $\mathbb{Q}$. It admits a Shimura datum $(\mathcal{G}_V, \mathcal{H}_V)$, where $\mathcal{H}_V$ is the union of the Siegel upper and lower half-spaces. Let $(G,X)$ be a Shimura datum of Hodge type with reflex field $E$ and let $(G,X) \to (\mathcal{G}_V, \mathcal{H}_V)$ be a Hodge embedding. Fix a prime $p>2$ such that $G_{\qp}$ is tamely ramified and such that the order of $\pi_1(\gder)$ is coprime to $p$. 

Choose a $\mathbb{Z}_{(p)}$-lattice $V_{(p)} \subset V$ on which $\psi$ is $\mathbb{Z}_{(p)}$-valued, and write $V_p=V_{(p)} \otimes \mathbb{Z}_p$. Write $\mathcal{K}_p \subset \mathcal{G}_V(\qp)$ for the stabiliser of $V_p$ in $\mathcal{G}_V(\qp)$, and similarly write $K_p$ for the stabiliser of $V_p$ in $G(\qp)$. We will assume that $K_p$ is a parahoric subgroup and let $\mathcal{G}$ be the parahoric group scheme over $\zp$ with $\mathcal{G}(\zp)=K_p$.

Because $K_p$ is the inverse image in $G(\qp)$ of the stabiliser of a lattice, it is automatically a \emph{connected parahoric subgroup} in the sense of \cite[start of Section 2]{Zhou}. Conversely, given a connected parahoric subgroup $K_p$ of $G(\qp)$ we can always find a Hodge embedding $(G,X) \to (\mathcal{G}_V, \mathcal{H}_V)$ and a $\mathbb{Z}_{(p)}$-lattice $V_{(p)} \subset V$ on which $\psi$ is $\mathbb{Z}_{(p)}$-valued, such that $K_p$ is the stabiliser of $V_p$ in $G(\qp)$; this is explained in \cite[Section 1.3.2]{KMPS}. 

\subsubsection{} For every sufficiently small compact open subgroup $K^p \subset G(\afp)$, we can find $\mathcal{K}^p \subset \mathcal{G}_V(\afp)$ such that the Hodge embedding induces a closed immersion (see \cite[Lemma 2.1.2]{KisinModels})
\begin{align}
    \mathbf{Sh}_{K}(G,X) \to \mathbf{Sh}_{\mathcal{K}}(\mathcal{G}_V, \mathcal{H}_V) \otimes_{\mathbb{Q}} E 
\end{align}
of Shimura varieties of levels $K=K^pK_p$ and $\mathcal{K}=\mathcal{K}^p \mathcal{K}_p$, respectively.
We let $\mathcal{S}_{\mathcal{K}}$ over $\mathbb{Z}_{(p)}$ be the moduli-theoretic integral model of $\mathbf{Sh}_{\mathcal{K}}(\mathcal{G}_V, \mathcal{H}_V)$; it is a moduli space of (weakly) polarised abelian schemes $(A, \lambda)$ up to prime-to-$p$ isogeny with level $\mathcal{K}^p$-structure. Fix a prime $v | p$ of $E$ and let
\begin{align}
    \mathscr{S}_{K}:=\mathscr{S}_K(G,X) \to \mathcal{S}_{\mathcal{K}} \otimes_{\mathbb{Z}_{(p)}} \mathcal{O}_{E,(v)}
\end{align}
be the normalisation of the Zariski closure of $\mathbf{Sh}_{K}(G,X)$ in $\mathcal{S}_{\mathcal{K}} \otimes_{\mathbb{Z}_{(p)}} \mathcal{O}_{E,(v)}$. This construction is compatible with changing the level away from $p$ and we define
\begin{align}
    \mathcal{S}_{\mathcal{K}_p}:=\varprojlim_{\mathcal{K}^p \subset \mathcal{G}_V(\afp)} \mathcal{S}_{\mathcal{K}^p \mathcal{K}_p}; \qquad
    \mathscr{S}_{K_p}:=\varprojlim_{K^p \subset G(\afp)} \mathscr{S}_{K^pK_p}.
\end{align}
Then, as discussed in \cite[Section 2.1]{KMPS}, the transition maps in both inverse systems are finite \'etale and moreover $G(\afp)$ acts on $\mathscr{S}_{K_p}$. Choose a map $\mathcal{O}_{E,(v)} \to \ovfp$ and write $\shginf$ for the base change to $\ovfp$ of $\mathscr{S}_{K_p}$ and $\shg$ for the base change to $\ovfp$ of $\mathscr{S}_{K^pK_p}$; these are both schemes over $\ovfp$ and $G(\afp)$ acts on $\shginf$. We will write $\shgsp$ for the base change to $\ovfp$ of $\mathcal{S}_{\mathcal{K}^p \mathcal{K}_p} \otimes_{\mathbb{Z}_{(p)}} \mathcal{O}_{E,(v)}$ and $\shgspinf$ the base change to $\ovfp$ of $\mathcal{S}_{\mathcal{K}_p} \otimes_{\mathbb{Z}_{(p)}} \mathcal{O}_{E,(v)}$. In particular, we are omitting $K_p$ from the notation.

Let $\pi:A \to \shgspinf$ be the universal abelian variety and let $\mathcal{V}^p$ be the prime-to-$p$ adelic Tate module of $A$; it is a smooth $\afp$-sheaf on $\mathcal{S}_{\mathcal{K}_p}$. As explained in \cite[Section 2.1.1]{KMPS}, there is a universal isomorphism
\begin{align}
    \epsilon:V \otimes \underline{\mathbb{A}}_f^p \simeq \mathcal{V}^p, 
\end{align}
sending the symplectic form $\psi$ to an $\underline{\mathbb{A}}_f^{p, \times}$-multiple of the Weil pairing. Here $\underline{\mathbb{A}}_f^p$ denotes the pro-\'etale sheaf associated to the topological group $\mathbb{A}_f^p$. 

\subsubsection{Tensors} \label{Sec:Tensors} Write $V^{\otimes}$ for the direct sum of $V^{\otimes n} \otimes (V^{\ast})^{\otimes m}$ for all pairs of integers $m \ge 0, n \ge 0$. We will also use this notation later for modules over commutative rings and modules over sheaves of rings. Write $G_{\mathbb{Z}_{(p)}}$ for the Zariski closure of $G$ in $GL(V_p)$, then $G_{\mathbb{Z}_{(p)}} \otimes \zp=\mathcal{G}$.

By \cite[Lemma 1.3.2]{KisinModels}, there are tensors $\{s_{\alpha}\}_{\alpha \in \mathscr{A}} \subset V_p^{\otimes}$ such that $G_{\mathbb{Z}_{(p)}}$ is their pointwise stabiliser in $\operatorname{GL}(V_p)$. Then, as explained in \cite[Section 1.3.4, Section 2.1.2]{KMPS}, there are global sections
\begin{align}
    \{s_{\alpha, \afp}\}_{\alpha \in \mathscr{A}} \in H^0( \mathscr{S}_{K^pK_p}, (\mathcal{V}^p)^{\otimes})
\end{align}
such that if we restrict the isomorphism $\epsilon$ via $\mathscr{S}_{K_p} \to \mathcal{S}_{\mathcal{K}_p}$, then we get an isomorphism
\begin{align}
    \eta:V \otimes \underline{\mathbb{A}}_f^p \to \mathcal{V}^p
\end{align}
taking $s_{\alpha} \otimes 1$ to $s_{\alpha, \afp}$ for all $\alpha \in \mathscr{A}$. In particular, for each $x \in \mathscr{S}_{K_p}(\ovfp)$, the stabiliser of the tensors $\{s_{\alpha, \afp,x}\}_{\alpha \in \mathscr{A}} \subset (\mathcal{V}^p_x)^{\otimes}$ in $\operatorname{GL}(\mathcal{V}^p_x)$ is canonically identified with $G \otimes \afp$. Here, the subscript $x$ denotes taking the stalk at $x$ of a sheaf (respectively, a section of a sheaf).

\subsubsection{Central leaves and Newton strata} \label{Sec:NewtonStratification} We will use $\zpbr$ to denote the $p$-typical Witt vectors $W(\ovfp)$ of $\ovfp$ and $\qpbr=\zpbr[1/p]$. We let $\sigma:\zpbr \to \zpbr$ be the automorphism induced by the Frobenius of $\ovfp$, and also denote by $\sigma$ the induced automorphism of $\qpbr$. Since $K_p$ is a parahoric subgroup, the integral models $\mathscr{S}_{\mathcal{K}^p \mathcal{K}_p}$ are the same as the ones constructed in \cite{KisinPappas}.

For $x \in \shg(\ovfp)$, we write $A_x$ for the abelian variety over $\ovfp$ corresponding to the image of $x \in \shgsp(\ovfp)$. Let $x \in \shg(\ovfp)$ and let $\mathbb{D}_x$ be the contravariant Dieudonn\'e module of the $p$-divisible group $A_x[p^{\infty}]$, equipped with its Frobenius $\phi$. By \cite[Paragraph before Proposition 2.4.2]{HamacherKim}, there are $\phi$-invariant tensors $\{s_{\alpha, \mathrm{cris},x} \}_{\alpha \in \mathscr{A}} \subset \mathbb{D}_x^{\otimes}$, and in \cite[Section 2.5]{HamacherKim} it is argued that there is an isomorphism $\zpbr \otimes V_p \to \mathbb{D}_x$ sending $1 \otimes s_{\alpha}$ to $s_{\alpha, \mathrm{cris},x}$ for all $\alpha \in \mathscr{A}$. 

Under such an isomorphism, the Frobenius $\phi$ corresponds to an element $b_x \in G(\qpbr)$, which is well-defined up to $\sigma$-conjugacy by $\mathcal{G}(\zpbr)$, where $\sigma:G(\qpbr) \to G(\qpbr)$ is induced by $\sigma:\qpbr \to \qpbr$. We will denote by $\llbracket b_x\rrbracket$ the $\mathcal{G}(\zpbr)$-$\sigma$-conjugacy class of $b_x$ and by $[b_x]$ the $G(\qpbr)$-$\sigma$-conjugacy class of $b_x$. We denote the set of $G(\qpbr)$-$\sigma$-conjugacy classes in $G(\qpbr)$ by $B(G)=B(G_{\qp})$. By \cite[Lemma 1.3.9]{KMPS}, the element $[b_x]$ is contained in the neutral acceptable set $B(G,\{\mu^{-1}\}) \subset B(G)$, consisting of the $\{\mu^{-1}\}$-admissible elements defined in \cite[Section 1.1.5]{KMPS}. 

There is a natural partial order on $B(G,\{\mu^{-1}\})$ defined in \cite{RapoportRichartz}. The set $B(G,\{\mu^{-1}\})$ admits a unique minimal element called the \emph{basic} element, and when $G_{\qp}$ is quasi-split it admits a unique maximal element called the $\mu$-ordinary element.

It follows from \cite[Cor. 3.3.8]{HamacherKim} that for $b \in G(\qpbr)$ there are (reduced) locally closed subschemes
\begin{align}
    \cb \subset \shgb \subset \shg
\end{align}
of $\shg$ such that their $\ovfp$-points can be identified with
\begin{align}
    \cb(\ovfp)&=\{x \in \shg(\ovfp) \; : \; \llbracket b_x\rrbracket=\llbracket b \rrbracket \} \\
    \shgb(\ovfp)&=\{x \in \shg(\ovfp) \; : \; [b_x]=[b] \}.
\end{align}
The subscheme $\shgb$ is called the \emph{Newton stratum} attached to $[b]$, and the subscheme $\cb \subset \shgb$ is called the \emph{central leaf} attached to $\llbracket b \rrbracket$. The construction of these subschemes is compatible with changing the level away from $p$ and we set
\begin{align}
    \cbinf&:=\varprojlim_{K^p \subset G(\afp)} \cb, \\
    \shgbinf&:=\varprojlim_{K^p \subset G(\afp)} \shgb.
\end{align}
Finally, we note that the natural map $\cb \to \shgb$ is a closed immersion by \cite[Cor. 3.3.8]{HamacherKim} and that the central leaf $\cb$ is smooth and equidimensional by \cite[Corollary 5.3.1]{KimLeaves}. 
\begin{Rem} \label{Rem:ChaiOortLeaves}
When $(G,X)=(\mathcal{G}_V, \mathcal{H}_V)$, then the $\mathcal{G}(\zpbr)$-conjugacy class $\llbracket b_x \rrbracket$ captures precisely the isomorphism class of the polarised $p$-divisible group $(A_x[p^{\infty}], \lambda_x)$, where an isomorphism of polarised $p$-divisible groups $f:(Y, \mu) \to (Y', \mu')$ is an isomorphism $f:Y \to Y'$ such that $f^{\ast} \mu'=c \mu$ for some $c \in \mathbb{Z}_p^{\times}$. In particular, when $(G,X)=(\mathcal{G}_V, \mathcal{H}_V)$ our central leaves do \emph{not} agree with those defined in \cite{ChaiOort}, which are defined using isomorphisms $f:(Y, \mu) \to (Y', \mu')$ with $f^{\ast} \mu'=\mu$. 

In general, our central leaves are finite unions of the central leaves of \cite{ChaiOort}, which explains why it is not surprising that our central leaves
can have more connected components than $\shg$, see Section \ref{counterexample}. It would be interesting to find an explicit example of this with $(G,X)=(\mathcal{G}_V, \mathcal{H}_V)$.
\end{Rem}
\subsubsection{} \label{Sec:TwistedCentralisers} For $b \in [b]$ we write $J_b$ for the twisted centraliser of $b$, which is an algebraic group over $\qp$ with $R$-points given by $$J_b(R)=\{g\in G(R\otimes_{\qp}\breve{\mathbb{Q}}_p)\; | \: g^{-1} b \sigma(g)=b\}.$$ The isomorphism class of the algebraic group $J_b$ only depends on the $\sigma$-conjugacy class $[b]$ of $b$. 

When $G_{\qp}$ is quasi-split, the algebraic group $J_b$ is an inner form of a standard Levi subgroup $M_b \subset G_{\qp}$ (see \cite[Section 1.1.4]{KMPS}). Moreover there is a natural map $J_{b,\qpbr} \to G_{\qpbr}$ with image $M_b$. Let $\gab_{\qp}$ denote the maximal abelian quotient of $G_{\qp}$, we will use the same notation for other reductive groups. The natural map $$M_b \to G_{\qp} \to \gab_{\qp}$$ factors through $M_b^{\mathrm{ab}}$, which we can identify with $J_b^{\mathrm{ab}}$ using the inner twisting. We denote the kernel of the composite map $J_b \to J_b^{\mathrm{ab}} \to \gab_{\qp}$ by $J_b'$.
\begin{Lem} \label{Lem:JbpConnected}
The group $J_b'$ is connected reductive.
\end{Lem}
\begin{proof}
We can prove this after basechanging to $\qpbr$ and we can identify $J_{b, \qpbr} \to \gab_{\qpbr}$ with $M_{b, \qpbr} \to \gab_{\qpbr}$. Thus it suffices to prove that the kernel $M_b'$ of $M_b \to \gab_{\qp}$ is connected. The group $M_b$ is connected reductive because it is a Levi subgroup, and $M_b'=\gder_{\qp} \cap M_b$ is connected because it is the corresponding Levi subgroup of $\gder$.
\end{proof}

If $\mathcal{G}$ is the parahoric group scheme over $\zp$ with $\mathcal{G}(\zp)=K_p \subset \mathcal{G}(\qp)=G(\qp)$, then we will write $H_{b} \subset J_b(\qp)$ for the compact open subgroup of $J_b(\qp)$ given by $\mathcal{G}(\zpbr) \cap J_b(\qp)$. 
\subsubsection{Automorphism groups} Let $x \in \shginf(\ovfp)$ and let $\operatorname{Aut}(A_x)$ be the algebraic group over $\mathbb{Q}$ with functor of points
\begin{align}
    R \mapsto \operatorname{Aut}(A_x)(R)=\left(\operatorname{End}(A_x) \otimes_{\mathbb{Z}} R \right)^{\times}.
\end{align}
Following \cite[Section 2.1.3]{KMPS}, we define $I_x^p$ to be the largest closed subgroup of $\operatorname{Aut}(A_x)$ that fixes the tensors $s_{\alpha, \afp,x}$ for all $\alpha \in \mathscr{A}$, and $I_x \subset I_x^p$ to be the largest closed subgroup that also fixes the tensors $s_{\alpha, \mathrm{cris},x}$ for all $\alpha \in \mathscr{A}$. Then there is a natural injective homomorphism of algebraic groups
\begin{align}
    j_{x}^p:I_{x,\afp} \to G_{\afp}.
\end{align}
The group $I_{x}$ is a connected reductive group and the subgroup $I_{x,\qlbar} \subset G_{\qlbar}$ is a Levi subgroup, see \cite[Corollary 2.1.9]{KMPS}. A choice of isomorphism $\qpbr \otimes V \to \mathbb{D}_x[1/p]$ sending $1 \otimes s_{\alpha}$ to $s_{\alpha,\mathrm{cris},x}$ for all $\alpha \in \mathscr{A}$, under which the Frobenius of $\mathbb{D}_x[1/p]$ corresponds to $b \in G(\qpbr)$, induces a map $I_{x,\qp} \to J_b$. 

\subsection{Igusa varieties}\label{Sec:Igusa} We will now recall the construction of Igusa varieties from \cite{HamacherKim}. Fix $[b] \in B(G,\{\mu^{-1}\})$ and let $\shgb \subset \shg$ be the corresponding Newton stratum. We also fix a basepoint $x \in \shgbinf(\ovfp)$ with corresponding principally polarised abelian variety $(A_x,\lambda_x)$ and write $(\mathbb{X},\mu)$ for the associated polarised $p$-divisible group over $\ovfp$. We also fix an isomorphism $\zpbr \otimes V_p \to \mathbb{D}_x$ sending $1 \otimes s_{\alpha}$ to $s_{\alpha,\mathrm{cris},x}$ for all $\alpha \in \mathscr{A}$, and we let $b \in G(\qpbr)$ be the element corresponding to the Frobenius of $\mathbb{D}_x$. Then we have the Igusa variety
\begin{align}
\iggsp \to \shgsp,
\end{align}
which is the $\shgsp$-scheme with functor of points 
\begin{align}
T \mapsto \operatorname{Isom}_{\lambda}((A_T[p^{\infty}],\lambda_T), (\mathbb{X}_T, \mu_T)).
\end{align}
Here $\operatorname{Isom}_{\lambda}$ denotes the set of isomorphisms $f:A_T[p^{\infty}] \to \mathbb{X}_T$ such that $f^{\ast} \mu_T = c \lambda_T$ for some $c \in \mathbb{Z}_p^{\times}$. The functor $\iggsp$ is representable by a perfect scheme by \cite[Proposition 4.3.3, Corollary 4.3.5]{CaraianiScholze}. The scheme $\iggsp$ has a natural (right) action of the profinite group $\operatorname{Aut}_{\lambda}(\mathbb{X})$ of automorphisms $f: \mathbb{X} \to \mathbb{X}$ satisfying $f^{\ast} \mu = c \mu$ for some $c \in \mathbb{Z}_p^{\times}$. Moreover, by \cite[Corollary 4.3.5]{CaraianiScholze}, the natural action $\operatorname{Aut}_{\lambda}(\mathbb{X})$ extends to an action of the locally profinite group $\operatorname{QIsog}_{\lambda}(\mathbb{X})$ of quasi-isogenies $f:\mathbb{X} \dashrightarrow \mathbb{X}$ satisfying $f^{\ast} \mu = c \mu$ for some $c \in \qp^{\times}$. \smallskip

In \cite[Section 5.1, Lemma 5.1.1]{HamacherKim}, Hamacher and Kim construct a perfect closed subscheme
\begin{align}
\ig \subset \iggsp \times_{\shgsp} \shg.
\end{align}
Its $\ovfp$-points consist precisely of those isomorphisms $f:A_y[p^{\infty}] \to \mathbb{X}=A_x[p^{\infty}]$, such that the induced isomorphism on Dieudonn\'e-modules
\begin{align}
    f:\mathbb{D}_y \to \mathbb{D}_x
\end{align}
satisfies $s_{\alpha, \mathrm{cris},y}=f^{\ast}s_{\alpha, \mathrm{cris},x}$ for all $\alpha \in \mathscr{A}$. In particular, it follows that the map $\ig \to \shg$ factors through $\cb$ (where $\llbracket b \rrbracket = \llbracket b_x \rrbracket$). There is a tautological point $\tilde{x} \in \ig(\ovfp)$, mapping to $x \in \cb(\ovfp)$, which corresponds to the identity map $A_x[p^{\infty}]=\mathbb{X} \to \mathbb{X}$. Hamacher and Kim prove (see \cite[Proposition 5.1.2]{HamacherKim}) that $\ig$ is stable under the action of the subgroup 
\begin{align}
\operatorname{QIsog}_{G}(\mathbb{X}) \subset \operatorname{QIsog}(\mathbb{X}),
\end{align}
consisting of those quasi-isogenies $f:\mathbb{X} \dashrightarrow \mathbb{X}$ such that the induced isomorphism $f:\mathbb{D}_x[1/p] \to \mathbb{D}_x[1/p]$ on rational Dieudonn\'e modules satisfies
\begin{align*}
f^{\ast}s_{\alpha, \mathrm{cris},x} = s_{\alpha, \mathrm{cris},x}
\end{align*}
for all $\alpha \in \mathscr{A}$. Using our fixed isomorphism $\zpbr \otimes V_p \to \mathbb{D}_x$ sending $1 \otimes s_{\alpha}$ to $s_{\alpha,\mathrm{cris},x}$ for all $\alpha \in \mathscr{A}$, we can identify
\begin{align}
\operatorname{QIsog}_{G}(\mathbb{X}) \simeq J_b(\qp).
\end{align}
Hamacher and Kim prove in \cite[Lemma 5.1.4]{HamacherKim} that their construction of $\ig$ is compatible with changing the level at $p$ and that $G(\afp)$ acts on
\begin{align}
\iginf:=\varprojlim_{K^p \subset G(\afp)} \ig
\end{align}
in a way that makes the map $\iginf \to \shginf$ into a $G(\afp)$-equivariant map.
\subsection{The product formula} \label{Sec:Product} Let the notation be as in Section \ref{Sec:Igusa}; in particular, we have a fixed base point $x \in \shgbinf(\ovfp)$, an isomorphism $\zpbr \otimes V_p \to \mathbb{D}_x$ sending $1 \otimes s_{\alpha}$ to $s_{\alpha, \mathrm{cris},x}$ for all $\alpha \in \mathscr{A}$, giving rise to $b=b_x \in G(\qpbr)$, and a tautological point $\tilde{x} \in \ig(\ovfp)$.

In \cite[Section 4.2]{HamacherKim}, the authors introduce a perfect scheme $\xmub$ equipped with an action of the locally profinite group $J_b(\qp)$, called an affine Deligne--Lusztig variety. To describe its set of $\ovfp$-points, we need to introduce some notations.

Let $\mathcal{G}$ be the parahoric group scheme over $\zp$ with $\mathcal{G}(\zp)=K_p$. Let $\{\mu\}$ be the $G(\qpbar)$-conjugacy class of cocharacters of $G$ from Section \ref{Sec:HodgeCocharacters} and let $\{\nu\}=\sigma ( \{\mu^{-1}\})$. Moreover, let $\admu \subset \tilde{W}$ be the admissible set inside the affine Weyl group $\tilde{W}$ of $G$ (see \cite[Section 4.1.2]{HamacherKim}). Then there is a $J_b(\qp)$-equivariant bijection (see \cite[Remark 4.2.1]{HamacherKim})
\begin{align} \label{Eq:DescriptionADLVOVFP}
    \xmub(\ovfp) \to \{g \cdot \mathcal{G}(\zpbr) \in G(\qpbr) / \mathcal{G}(\zpbr) : g^{-1} b \sigma(g) \in \bigcup_{w \in \admu} \mathcal{G}(\zpbr) w \mathcal{G}(\zpbr) \},
\end{align}
where the right-hand side is equipped with the action of $J_b(\qp) \subset G(\qpbr)$ by left translation. It is explained in \cite[Section 6.7]{Zhou} that $b \in \mathcal{G}(\zpbr) w \mathcal{G}(\zpbr)$ for some $w \in \admu$; in particular $1 \in \xmub(\ovfp)$. For an element $y \in \xmub(\ovfp)$ of the form $g_y \cdot \mathcal{G}(\zpbr)$, we will write $b_y$ for the element $g_y^{-1} b \sigma(g_y) \in G(\qpbr)$. 
\begin{Lem} \label{Lem:CentralLeavesShtukas}
Two points $y,y' \in \xmub(\ovfp)$ satisfy $\llbracket b_y \rrbracket = \llbracket b_{y'} \rrbracket$ if and only if $y'$ is in the $J_b(\qp)$-orbit of $y$.
\end{Lem}
\begin{proof}
This is a straightforward consequence of the definition of $J_b(\qp)$ and the description of $\xmub(\ovfp)$ in \eqref{Eq:DescriptionADLVOVFP}. \iffalse
Write $y=g_y \cdot \mathcal{G}(\zpbr)$. Note that $h \in J_b(\qp)$ satisfies $h^{-1} b \sigma(h)=b$ per definition, and so 
\begin{align}
    b_{h \cdot y} = g_y^{-1} h^{-1} b \sigma(h) \sigma(g_y) = g_y^{-1} b \sigma(g_y) = b_{y}.
\end{align}
For the converse direction we write $y'=g_y' \cdot \mathcal{G}(\zpbr)$. If there is an element $g \in G(\zpbr)$ such that
\begin{align}
    g^{-1} g_y^{-1} b \sigma(g_y) \sigma(g) = g_{y'}^{-1} b \sigma(g_{y'}),
\end{align}
then
\begin{align}
    g_{y'} g^{-1} g_y^{-1} b \sigma(g_y) \sigma(g) \sigma(g_{y'})^{-1}=  b ,
\end{align}
and so $h=g_y g g_{y'}^{-1} \in J_b(\qp)$. But $g_y g \mathcal{G}(\zpbr) = g_y \mathcal{G}(\zpbr)$ and so after replacing $g_y$ by $g_y g$ we see that
\begin{align}
    h g_{y}' = g_{y}.
\end{align}
\fi
\end{proof}
The stabiliser $H_y$ of a point $g_y \mathcal{G}(\zpbr)=y \in \xmub(\ovfp)$ is given by the intersection of $J_b(\qp)$ with the compact open subgroup $g_y \mathcal{G}(\zpbr) g_y^{-1}$ and therefore $H_y \subset J_b(\qp)$ is a compact open subgroup. Left multiplication by $g_y$ induces an isomorphism $J_b \to J_{b_y}$ that identifies $H_y$ with $H_{b_y}:=J_{b_y}(\qp) \cap \mathcal{G}(\zpbr)$. 
\begin{Lem} \label{Lemma:OrbitClosedADLV}
Let $y \in \xmub(\ovfp)$ and let $\operatorname{Orb}(y) \subset \xmub(\ovfp)$ be the $J_b(\qp)$-orbit of $y$. Then $\operatorname{Orb}(y)$ is Zariski closed inside $\xmub(\ovfp)$.
\end{Lem}
\begin{proof}
The stabiliser $H_Z$ of an irreducible component $Z$ of $\xmub$ containing $y$ is a parahoric subgroup of $J_b(\qp)$ by \cite[Theorem 3.1.1]{ZhouZhu}, and hence contains $H_y$ as a finite index subgroup. Thus the fibers of 
\begin{align}
    J_b(\qp)/H_y \to J_b(\qp)/H_Z
\end{align}
are finite, or equivalently, the orbit of $y$ under $J_b(\qp)$ intersects each irreducible component in the orbit of $Z$ in finitely many points. We conclude that $\operatorname{Orb}(y)$ intersects each irreducible component $Z$ of $\xmub$ in finitely many points.

This implies that $\operatorname{Orb}(y) \cap Z$ is closed in $Z$ for all irreducible components $Z$ of $\xmub$. Since $\xmub$ has an open cover by perfections of finite type $\ovfp$-schemes, which have finitely many irreducible components, we deduce from this that $\operatorname{Orb}(y)$ is closed inside all of $\xmub$. 
\end{proof}

\subsubsection{} \textbf{From now on, we will assume that $G_{\qp}$ is quasi-split and that $K_p$ is a very special parahoric subgroup.} In \cite[Section 5.2, Theorem 5.2.6]{HamacherKim}, Hamacher and Kim construct a $J_b(\qp)$-invariant and $G(\afp)$-equivariant map, known as the product formula: 
\begin{align}
\pi_{\infty}:\iginf \times \xmub \to \shgbinf^{\mathrm{perf}},
\end{align}
where $\mathrm{perf}$ denotes the perfection of a scheme. The construction of this map relies on \cite[Axiom A]{HamacherKim}, which is true under our assumptions by \cite[Theorem 2]{vH}, see also \cite[Corollary 1.6]{GleasonLimXu}. Furthermore, Hamacher and Kim prove that this map is a $J_b(\qp)$-torsor for the pro-\'etale topology, see \cite[Corollary 5.2.7]{HamacherKim}. We will also write $\pi_{\infty}$ for the induced map at level $K^p \subset G(\afp)$. It follows from the construction that $\pi_{\infty}(\tilde{x},1)=x$, where $\tilde{x} \in \ig(\ovfp)$ is the tautological point as in the beginning of Section \ref{Sec:Product}.
\begin{Prop} \label{Prop:ProductFormulaLemma}
If $y \in \xmub(\ovfp)$ is a point with stabiliser $H_y \subset J_b(\qp)$, then the map
\begin{align}
\ig \times \operatorname{Orb}(y) \to \shgb^{\mathrm{perf}}
\end{align}
factors through $\cby^{\mathrm{perf}} \subset \shgb^{\mathrm{perf}}$. Moreover, the following diagram is Cartesian:
\begin{equation}
    \begin{tikzcd}
    \ig \times \operatorname{Orb}(y) \arrow{d} \arrow{r} & \ig \times \xmub \arrow{d} \\
    \cby^{\mathrm{perf}} \arrow{r} & \shgb^{\mathrm{perf}}.
    \end{tikzcd}
\end{equation}
\end{Prop}
\begin{proof}
We start by pointing out that both $\ig$ and $\xmub$ have a dense set of $\ovfp$-points. This is true for $\ig$ because it is an inverse limit of perfections of finite type $\ovfp$-schemes along finite \'etale transition maps, and for $\xmub$ since it has an open cover by perfections of finite type $\ovfp$-schemes. Then, since perfect schemes are reduced, it suffices to prove the first claim on the level of $\ovfp$-points. In this case, what we want to prove is that for $w=\pi_{\infty}(z,y)$ we have an equality $\llbracket b_w \rrbracket=\llbracket b_y \rrbracket$.

By \cite[Lemma 5.2.5]{HamacherKim}, the restriction of $\pi_{\infty}$ to $\{z \} \times \xmub$, for $z \in \ig(\ovfp)$, can be identified with the so-called Rapoport--Zink uniformisation map $\Theta_{z}:\xmub \to \shgb$; see \cite[Section 4.3]{HamacherKim}. The relevant result for the Rapoport--Zink uniformisation map is \cite[Proposition 6.5]{Zhou}, see also the discussion in \cite[Section 8. Axiom 4]{Zhou}. \smallskip

It follows from this discussion that the diagram is Cartesian on the level of $\ovfp$-points, since $\llbracket b_y \rrbracket =\llbracket b_{y'} \rrbracket$ if and only if $y$ and $y'$ are in the same $J_b(\qp)$-orbit in $\xmub(\ovfp)$; this is Lemma \ref{Lem:CentralLeavesShtukas}. We claim that this implies that the diagram is Cartesian on the level of perfect schemes. Indeed, both the fiber product and the image of $\ig \times \operatorname{Orb}(y) $ are reduced closed subschemes of $\ig \times \xmub$, see Lemma \ref{Lemma:OrbitClosedADLV}, and reduced closed subschemes are determined by their (dense sets of) $\ovfp$-points. 
\end{proof}
\begin{Cor} \label{Cor:IgusaVarietiesAsTorsors}
For $y \in \xmub(\ovfp)$, the natural map
\begin{align}
    \pi_{\infty}(-,y):\ig \times \{y\} \to \cby
\end{align}
is a pro-\'etale $H_y$-torsor.
\end{Cor}
\begin{proof}
It follows from Proposition \ref{Prop:ProductFormulaLemma} that
\begin{align}
\ig \times \operatorname{Orb}(y) \to \cby^{\mathrm{perf}}
\end{align}
is a pro-\'etale $J_b(\qp)$-torsor, and therefore the induced map
\begin{align*}
\ig \times \{y\} \to \cby^{\mathrm{perf}}
\end{align*}
is a pro-\'etale $H_y=\operatorname{Stab}_{J_b(\qp)}(y)$-torsor.
\end{proof}
Since the product formula is $G(\afp)$-equivariant it follows that the maps $\ig \times \{y\} \to \cby$ are compatible with changing the level away from $p$. Therefore there is an induced $G(\afp)$-equivariant map
\begin{align}
    \iginf \to \cbyinf^{\mathrm{perf}},
\end{align}
which is again a pro-\'etale $H_y$-torsor. 
\begin{Lem} \label{Lem:Continuity}
The locally profinite group $J_b(\qp)$ acts continuously on $\pi_0(\ig)$.
\end{Lem}
\begin{proof}
The Igusa variety $\ig \to \cby^{\mathrm{perf}}$ is a pro-\'etale $H_y$-torsor by Corollary \ref{Cor:IgusaVarietiesAsTorsors} and therefore $H_y$ acts continuously on $\pi_0(\ig)$, see \cite[Lemma 3.1.4, Corollary 3.1.5]{OrdinaryHO}. The lemma now follows because $H_y \subset J_b(\qp)$ is an open subgroup, and thus if the $H_y$ action is continuous, then so is the $J_b(\qp)$-action.
\end{proof}
\subsection{Connected components and the product formula} \label{Sec:ProductFormulaAndPiNaught} We start with a lemma about connected components. Let $\nu:G \to \gab$ be the natural map and \textbf{assume from now on that $\gder$ is simply connected.} 
\begin{Lem} \label{Lem:ConnectedComponents}
The map
\begin{align}
    \pi_0(\mathbf{Sh}_{K_p}(G,X)) \to \pi_0(\mathbf{Sh}_{K_pK^p}(G,X))
\end{align}
is a $\nu(K^p)$-torsor, compatible with the action of $K^p$ on $\mathbf{Sh}_{K_p}(G,X)$.
\end{Lem}

\begin{proof}
Let $\gab(\mathbb{R})^{\dagger}$ be the image of $Z_G(\mathbb{R}) \to \gab(\mathbb{R})$ and let $\gab(\mathbb{Q})^{\dagger}$ be its intersection with $\gab(\mathbb{Q})$. 
By \cite[Theorem 5.17]{Milne}, there is a natural identification
\begin{align}
    \pi_0(\mathbf{Sh}_{K^p K_p}(G,X)) = \gab(\mathbb{Q})^{\dagger} \backslash \gab(\mathbb{A}_f) / \nu(K^pK_p),
\end{align}
compatible with changing $K^p$, where $\nu:G(\mathbb{A}_f) \to \gab(\mathbb{A}_f)$ is the natural map. Since $(G,X)$ is of Hodge type, axiom SV5 of op. cit. is satisfied for $(G,X)$ and therefore also for $(\gab, X^{\mathrm{ab}})$. Therefore it follows from \cite[Lemma 4.20]{Milne} that there is a bijection
\begin{align}
    \pi_0(\mathbf{Sh}_{K_p}(G,X)) &= \gab(\mathbb{Q})^{\dagger} \backslash \gab(\mathbb{A}_f) / \nu(K_p).
\end{align}
We see that the map 
\begin{align}
    \pi_0(\mathbf{Sh}_{K_p}(G,X)) \to \pi_0(\mathbf{Sh}_{K_pK^p}(G,X))
\end{align}
is a pro-\'etale $\nu(K^p)$-torsor, compatible with the action of $K^p$ on $\mathbf{Sh}_{K_p}(G,X)$.
\end{proof}
\begin{Lem} \label{Lem:ConnectedComponentSpecialFibre}
For each finite extension $F$ of the reflex field $E$ and any place $w$ of $F$ extending $v$, the natural maps
\begin{align}
    \pi_0(\mathbf{Sh}_{K^pK_p}(G,X) \otimes_E F) \leftarrow \pi_0(\mathscr{S}_{K^pK_p} \otimes_{\mathcal{O}_{E,(v)}} \mathcal{O}_{F,(w)}) \to \pi_0(\mathscr{S}_{K^pK_p} \otimes_{\mathcal{O}_{E,(v)}} k(w))
\end{align}
are isomorphisms. 
\end{Lem}
\begin{proof}
The Shimura variety $\shg$ is locally integral since $K_p$ is very special, see \cite[Corollary 4.6.26]{KisinPappas}. The result now follows from \cite[Corollary 4.1.11]{MP}.
\end{proof}

\subsubsection{} Let $\pi_1(G)$ be the algebraic fundamental group of $G \otimes \qpbar$, see \cite[Definition 3]{Borovoi}. Let $\pi_1(G)_I$ be the coinvariants under the action of the inertia group $I=\gal(\qpbar/\mathbb{Q}_p^{\mathrm{ur}})$, and let $\pi_1(G)_I^{\sigma}$ be the invariants of $\pi_1(G)_I$ under Frobenius. Recall the functorial Kottwitz homomorphism from \cite[Chapter 11.5]{KalethaPrasad}
\begin{align}
    \kappa:G(\qpbr) \to \pi_1(G)_I.
\end{align}
By \cite[Lemma 3.4.2]{vH}, the restriction of $\kappa$ to $G(\qp)$ identifies 
\begin{align}
    \frac{G(\qp)}{\gder(\qp) K_p} \to \frac{\gab(\qp)}{\nu(K_p)} \simeq \pi_1(G)_I^{\sigma}.
\end{align}
It follows from the proof of Lemma \ref{Lem:ConnectedComponents} that the abelian group $\gab(\qp)/\nu(K_p)=\pi_1(G)_I^{\sigma}$ acts on $\pi_0(\mathbf{Sh}_{K_p}(G,X))$, and that the action commutes with that of $G(\afp)$. Using the isomorphisms from Lemma \ref{Lem:ConnectedComponentSpecialFibre}, this gives an action of $\pi_1(G)_I^{\sigma} \times G(\afp)$ on $\pi_0(\shginf)$.
\begin{Lem} \label{Lem:UniformisationAndPiNaught}
If $\tilde{x} \in \ig(\ovfp)$ is the tautological point as in the beginning of Section \ref{Sec:Product}, then 
\begin{align}
    \Theta_{\tilde{x}}:=\pi_{\infty}(\tilde{x},-):\xmub \to \shgb \to \pi_0(\shg)
\end{align}
is $J_b(\qp)$-equivariant, where $J_b(\qp)$ acts on the target via the natural map $J_b(\qp) \to \gab(\qp) \to \pi_1(G)_I^{\sigma}$.
\end{Lem}
\begin{proof}
This is \cite[Proposition 3.4.5]{vH}, where $\Theta_{\tilde{x}}$ is denoted by $i_x$. Indeed, it is shown there that the point $i_x(y)$ lies in the connected component $\kappa(y) \cdot i_x(1)$, where $\kappa:\pi_0(\xmub) \to \pi_1(G)_I^{\sigma}$ is the natural map introduced in \cite[Section 3.4.4]{vH}. This natural map is $J_b(\qp)$-equivariant for the $J_b(\qp)$ action on $\pi_1(G)_I^{\sigma}$ coming from the natural map $J_b(\qp) \to \gab(\qp) \to \pi_1(G)_I^{\sigma}$; this proves the lemma. 
\end{proof}

\subsection{Distinguished central leaves} \label{Sec:Distinguished} Let $\tilde{W}$ be the affine Weyl group of $G$. Recall from \cite[Lemma 2.2.8, Definition 2.2.9]{KimLeaves} that to an element $w \in \tilde{W}$ we can associate a well-defined $\mathcal{G}(\zpbr)$-$\sigma$-conjugacy class $\llbracket w \rrbracket$. Recall from \cite[Section 1.2.10]{ShenYuZhang} the notion of a \emph{$\sigma$-straight} element of $\tilde{W}$.
\begin{Def} \label{Def:Distinguished}
We call $y \in \xmub(\ovfp)$ \emph{distinguished} if $\llbracket b_y \rrbracket=\llbracket w \rrbracket$ for some $\sigma$-straight element $w \in \admu$. We call a central leaf $\cb$ \emph{distinguished} if $\llbracket b \rrbracket=\llbracket w \rrbracket$ for some $\sigma$-straight element $w \in \admu$.
\end{Def}
\begin{Lem}
If $\cb$ is distinguished, then $\cb \subset \shgb$ is an Ekedahl--Kottwitz--Oort--Rapoport (EKOR) stratum in the sense of \cite[Theorem 3.4.12]{ShenYuZhang}.
\end{Lem}
\begin{proof}
Write $\llbracket b \rrbracket = \llbracket w \rrbracket$ with $w \in \admu$ a $\sigma$-straight element. By the proof of \cite[Theorem 6.17]{He-Rapoport}, there is an element $v \in \tilde{W}$ such that $w':= v w \sigma(v)^{-1}$ lies in $^{K} \admu :={}^K \tilde{W} \cap \admu$ and $w'$ is again a $\sigma$-straight element. (The subset ${}^K \tilde{W} \subset \tilde{W}$ is introduced in \cite[page 3125]{ShenYuZhang}, but its precise meaning is not relevant for us.) Then $\llbracket w \rrbracket=\llbracket w' \rrbracket$ and the result now follows from \cite[Corollary 3.4.14]{ShenYuZhang}, see \cite[paragraph after Theorem 1.3.5]{ShenYuZhang}.
\end{proof}
\begin{Lem} \label{Lem:ExistenceDistinguished}
There exists a distinguished central leaf $\cb \subset \shgb$. 
\end{Lem}
\begin{proof}
This follows from \cite[Theorem 1.3.5, paragraph after Theorem 1.3.5]{ShenYuZhang}.
\end{proof}
\begin{Lem} \label{Lem:IsParahoric}
If $y$ is distinguished, then there are parahoric subgroups $J$ and $J'$ of $J_b(\qp)$ such that $J \subset H_y \subset J'$.
\end{Lem}
\begin{proof}
The stabiliser $H_y \subset J_b(\qp)$ of the point $y \in \xmub(\ovfp)$ is contained in the stabiliser $H_Z$ of an irreducible component $Z$ containing $y$. The subgroup $H_Z \subset J_b(\qp)$ is a parahoric subgroup of $J_b(\qp)$ by \cite[Theorem 3.1.1]{ZhouZhu}, thus we can take $J'=H_Z$.

If $y=g_y \cdot \mathcal{G}(\zpbr)$ is distinguished, then we can find a representative of $g_y$ such that $g_y^{-1} b \sigma(g_y) = w$, where $w \in \tilde{W}$ is a $\sigma$-straight element. Let $\mathcal{I}(\zpbr)$ be a standard Iwahori subgroup containing $\mathcal{G}(\zpbr)$, then as explained in Section 5.3 of \cite{He-Zhou}, the $\sigma$-centraliser $J_w(\qp) \subset G(\qpbr)$ intersects $\mathcal{I}(\zpbr)$ in an Iwahori subgroup of $J_w(\qp)$. After conjugating by $g_y$, we see that $H_y$ contains an Iwahori subgroup $J$ of $J_b(\qp)$.
\end{proof}
\subsection{Some results on parahoric group schemes} \label{Sec:Parahoric} Let $K_p \subset G(\qp)$ be a parahoric subgroup and let $J_p \subset J_b(\qp)$ be a parahoric subgroup. Let $\mathcal{Z}$ denote the connected N\'eron model of $\gab_{\qp}$ over $\zp$ and let $\gab(\zp):=\mathcal{Z}(\zp)$. The goal of this section is to prove the following results, the second of which is well-known.
\begin{Prop} \label{Prop:ImageParahoricI}
The image of $J_p \to J_b(\qp) \to \gab(\qp)$ is equal to $\gab(\zp)$.
\end{Prop}
\begin{Prop} \label{Prop:ImageParahoricII}
The image of $K_p \to G(\qp) \to \gab(\qp)$ is equal to $\gab(\zp)$.
\end{Prop}
We will refer to \cite[Chapter 4]{KalethaPrasad} for conventions regarding Bruhat--Tits buildings, parahoric subgroups and parahoric group schemes. In particular, we do not work with the extended Bruhat--Tits buildings in this paper. Thus, for a parahoric subgroup $K_p \subset G(\qp)$, there is a point $x$ of the Bruhat--Tits building $\mathcal{B}(G_{\qp})$ such that $K_p=\mathcal{G}_{x}^0(\zp)$. Here the Bruhat--Tits parahoric group scheme $\mathcal{G}_{x}^0$ is a smooth affine group scheme which is the relative identity component of the Bruhat--Tits stabiliser group scheme $\mathcal{G}_{x}^1$. We will similarly write $J_p=\mathcal{J}_{x_J}^0(\zp)$ for a point $x_J$ of the Bruhat--Tits building of $\mathcal{B}(J_b)$ of $J_b$.\smallskip 

Recall that $J_b'$ is a connected reductive group of $J_b$ whose derived group is isomorphic to $\jbder$, see Lemma \ref{Lem:JbpConnected}. Since the buildings of $\gder_{\qp}$ and $G_{\qp}$ respectively of $J_b'$ and $J_b$ are equal, see \cite[bottom of page 343]{KalethaPrasad}, the inclusions $\gder_{\qp} \to G_{\qp}$ and $J_b' \to J_b$ induce morphisms
\begin{align}
    \mathcal{G}^{\mathrm{der},1}_{x} \to \mathcal{G}_{x}^1 \\
    \mathcal{J}^{',1}_{x_J} \to \mathcal{J}_{x_J}^1.
\end{align}
By \cite[Proposition 2.4.8]{KisinZhou}, these maps are closed immersions since $G_{\qp}$ splits over a tamely ramified extension of $\qp$ (and thus the same holds for $J_{b}$). Note that the superscript $1$ in our notation corresponds to the tilde in the notation of \cite[Proposition 2.4.9]{KisinZhou}.
\begin{Lem}\label{Lem:Connected}
The preimage of $\mathcal{J}_{x_J}^0$ in $\mathcal{J}^{',1}_{x_J}$ has connected special fibre. In particular, it is equal to $\mathcal{J}^{',0}_{x_J}$.
\end{Lem}
\begin{proof}
Recall that there is a functorial Kottwitz homomorphism
\begin{align}
    \kappa:J_{b}'(\qpbr) \to \pi_1(J_b')_I.
\end{align}
Here $\pi_1(J_b')$ is the algebraic fundamental group of $J_b'$ and $I$ is the inertia group. In particular, there is a commutative diagram
\begin{equation} \label{Eq:KottwitzHomomorphism}
    \begin{tikzcd}
    \mathcal{J}^{',1}_{x_J}(\zpbr) \arrow{d} \arrow{r} & \mathcal{J}_{x_J}^1(\zpbr)  \arrow{d} \\
    \pi_1(J_b')_I \arrow{r} & \pi_1(J_b)_I.
    \end{tikzcd}
\end{equation}
By \cite[Corollary 11.6.3]{KalethaPrasad}, the images of the horizontal maps in \eqref{Eq:KottwitzHomomorphism} can be identified with the component groups of the special fibers of $\mathcal{J}^{',1}_{x_J, \zpbr}$ and $\mathcal{J}_{x_J, \zpbr}^1 $, respectively. Thus to show that the inverse image of $\mathcal{J}_{x_J}^0$ in $\mathcal{J}^{',1}_{x_J}$ is connected, it would be enough to show that $\pi_1(J_b')_I \to \pi_1(J_b)_I$ is injective. 

To prove this, we need to recall the definition of the algebraic fundamental group. Let $T \subset J_{b,\qpbr}$ be a maximal torus defined over $\qp$ that is the centraliser of a maximal $\qpbr$-split torus, let $T'$ be its intersection with $J_{b,\qpbr}'$ and $T^{\mathrm{der}}$ its intersection with $J_{b,\qpbr}^{\mathrm{der}}$. Then the short exact sequences defining the algebraic fundamental groups are given by 
\begin{equation}
    \begin{tikzcd}
    0 \arrow{r}& X_{\ast}(T^{\mathrm{der}}) \arrow[d, equals] \arrow{r} & X_{\ast}(T') \arrow{d} \arrow{r} & \pi_1(J_b') \arrow{d} \arrow{r} & 0 \\
    0 \arrow{r} & X_{\ast}(T^{\mathrm{der}}) \arrow{r} & X_{\ast}(T) \arrow{r} & \pi_1(J_b) \arrow{r} & 0,
    \end{tikzcd}
\end{equation}
where $X_{\ast}$ denotes the cocharacter lattice of a torus, equipped with its natural action of $I$. Recall that $J_{b,\qpbr} \simeq M_{b,\qpbr} \subset G_{\qpbr}$ is a standard Levi. In particular, $T \subset J_{b,\qpbr} \subset G_{\qpbr}$ is also the centraliser of a maximal $\qpbr$-split torus in $G_{\qpbr}$. Then $X_{\ast}(T^{\mathrm{der}})$ and $X_{\ast}(T')$ are both induced Galois modules by \cite{BTII}*{Proposition 4.4.16}, since $\gder$ and $\jbder$ are simply connected (see \cite{MalleTesterman}*{Proposition 12.14}). This means that $X_{\ast}(T^{\mathrm{der}})_I$ and $X_{\ast}(T')_I$ are torsion-free. In particular, the maps
\begin{align}
    X_{\ast}(T^{\mathrm{der}})_I \to X_{\ast}(T')_I \text{ and }
    X_{\ast}(T^{\mathrm{der}})_I \to X_{\ast}(T)_I,
\end{align}
which are injective after tensoring with $\mathbb{Q}$, are injective. This gives us a diagram of short exact sequences
\begin{equation}
    \begin{tikzcd}
    0 \arrow{r}& X_{\ast}(T^{\mathrm{der}})_I \arrow[d, equals] \arrow{r} & X_{\ast}(T')_I \arrow{d}{b} \arrow{r} & \pi_1(J_b')_I \arrow{d}{c} \arrow{r} & 0 \\
    0 \arrow{r} & X_{\ast}(T^{\mathrm{der}})_I \arrow{r} & X_{\ast}(T)_{I} \arrow{r} & \pi_1(J_b)_I \arrow{r} & 0, 
    \end{tikzcd}
\end{equation}
and the snake lemma gives us an isomorphism $\ker c \simeq \ker b$. Because $X_{\ast}(T')_I$ is torsion free, it follows that $\ker c$ is trivial, since $c$ is injective after tensoring with $\mathbb{Q}$.
\end{proof}
\begin{Lem} \label{Lem:ConnectedII}
We have an equality $\mathcal{G}^{\mathrm{der},1}_{x}=\mathcal{G}^{\mathrm{der},0}_{x}$.
\end{Lem}
\begin{proof}
This can be proved as in the proof of Lemma \ref{Lem:Connected}, using the fact that $\pi_1(\gder)=0$ because $\gder$ is semisimple and simply connected, see \cite[Example 1.6]{Borovoi}.
\end{proof}
Let $\mathcal{Z}$ denote the connected N\'eron model of $\gab_{\qp}$ as before.
\begin{Lem} \label{Lem:ParahoricAbelianQuotient}
There is a short exact sequence
\begin{align} 
    &1 \to \mathcal{J}^{',0}_{x_J} \to \mathcal{J}^{0}_{x_J} \to \mathcal{Z} \to 1.
\end{align} 
\end{Lem}
\begin{proof}
There is a natural group homomorphism from $\mathcal{J}^{0}_{x_J}$ to the lft-N\'eron model of $\gab_{\qp}$, by the universal property of the lft-N\'eron model of tori. Since $\mathcal{J}^{0}_{x_J}$ has connected special fiber, it follows that this group homomorphism lands in $\mathcal{Z}$. To show that the map $\mathcal{J}^{0}_{x_J} \to \mathcal{Z}$ is surjective, we argue as in \cite{KisinPappas}*{Proposition 1.1.4}. There, it is argued that there exists some tamely ramified maximal torus $T$ of $J_b$ whose connected Néron model $\mathcal{T}$ is contained in $\mathcal{J}^{0}_{x_J}$. It then follows from \cite[Lemma 6.7]{PappasRapoportAffineFlag} that $\mathcal{T} \to \mathcal{Z}$ is surjective.

We are left to show that $\mathcal{J}^{',0}_{x_J}$ is the kernel of the map $\mathcal{J}^{0}_{x_J} \to \mathcal{Z}$. This can be checked on $\zpbr$-points, and there we consider the following diagram
\begin{equation}
    \begin{tikzcd}
     & \mathcal{J}^{',0}_{x_J}(\zpbr) \arrow[d, hook] \arrow{r} & \mathcal{J}^{0}_{x_J}(\zpbr) \arrow{r} \arrow[d, hook] & \mathcal{Z}(\zpbr) \arrow{r} \arrow[d, hook] &1  \\ 
     1 \arrow{r} & J_b'(\qpbr) \arrow{r} & J_b(\qpbr) \arrow{r} & \gab(\breve{\mathbb{Q}}_p) \arrow{r} & 1.
    \end{tikzcd}
\end{equation}
The exactness on the left of the top row follows from the fact that the leftmost square is Cartesian, see Lemma \ref{Lem:Connected}, and so we are done. 
\end{proof}
\begin{Lem} \label{Lem:ParahoricAbelianQuotientII}
There is a short exact sequence
\begin{align} \label{Eq:Parahoric}
    &1 \to \mathcal{G}^{\mathrm{der},0}_{x} \to \mathcal{G}_{x}^0 \to \mathcal{Z} \to 1.
\end{align} 
\end{Lem}
\begin{proof}
This can be established by a similar, but simpler version of the proof of Lemma \ref{Lem:ParahoricAbelianQuotient}.
\end{proof}
\begin{proof}[Proof of Proposition \ref{Prop:ImageParahoricI}]
Taking $\zp$-points of the short exact sequence from Lemma \ref{Lem:ParahoricAbelianQuotient}, we see that it suffices to show that $H^1(\zp, \mathcal{J}^{',0}_{x_J})=0$. But this follows from Lang's Lemma, since $\mathcal{J}^{',0}_{x_J}$ is a smooth group scheme with connected special fiber.
\end{proof}
The proof of Proposition \ref{Prop:ImageParahoricII} is the same as the Proof of Proposition \ref{Prop:ImageParahoricI}, using Lemma \ref{Lem:ParahoricAbelianQuotientII} instead of Lemma \ref{Lem:ParahoricAbelianQuotient}.

\section{Geometric Monodromy} \label{Sec:Monodromy}
Let the notation be as in Section \ref{Sec:Notation} and recall that we have assumed that $K_p$ is a very special parahoric and that $\gder$ is simply connected. Moreover, recall that we have a fixed base point $x \in \shgbinf(\ovfp)$ with tautological point $\tilde{x} \in \ig(\ovfp)$, and an isomorphism $\zpbr \otimes V_p \to \mathbb{D}_x$ sending $1 \otimes s_{\alpha}$ to $s_{\alpha, \mathrm{cris},x}$ for all $\alpha \in \mathscr{A}$, giving rise to $b=b_x \in G(\qpbr)$. By Lemma \ref{Lem:ExistenceDistinguished} we may choose $x$ to lie in a distinguished central leaf, and \textbf{we will assume from now on that $\cb$ is distinguished.} Then the central leaf $\cb$ is equal to $\cbynaught$ for $1 \in \xmub(\ovfp)$, and we write $H:=H_{1}$. 

Consider the product decomposition
$G^{\mathrm{ad}}=\prod_{i=1}^n G_i$ into simple groups over $\mathbb{Q}$, which induces maps
\begin{align} \label{eq:AdmissibleAdjoint}
    B(G_{\qp}) \to B(G^{\mathrm{ad}}_{\qp}) \to \prod_{i=1}^n B(G_{i,\qp}).
\end{align}
For an element $[b] \in B(G_{\qp})$ we will write $[b_i]$ for its image in $B(G_{i,\qp})$ under this map. Recall from \cite[Def. 5.3.2]{KretShin} that an element $[b] \in B(G_{\qp})$ is called \emph{$\mathbb{Q}$-non-basic} if $[b_i]$ is non-basic for all $i$. A Newton stratum $\shgb$ is called \emph{$\mathbb{Q}$-non-basic} if $[b]$ is $\mathbb{Q}$-non-basic. \textbf{Assume from now on that $[b]$ is $\mathbb{Q}$-non-basic.}

\subsection{\texorpdfstring{$\ell$}{ell}-adic monodromy} Recall that $\shginf \to \shg$ is a pro-\'etale $K^p$-torsor. Let $\ell \not=p$ be a prime number, let $K_{\ell}$ be the image of the projection $K^p \to G(\ql)$, and let $\shginfl \to \shg$ be the induced pro-\'etale $K_{\ell}$-torsor.

Let $\pi:A \to \shg$ be the abelian scheme pulled back from the universal abelian variety up to prime-to-$p$ isogeny over $\shgsp$ along $\shg \to \shgsp$. The local system $R^1 \pi_{\ast} \ql$ corresponds to the pro-\'etale $\operatorname{GL}_V(\ql)$-torsor over $\shg$ given by pushing out $\shginfl \to \shg$ via $K_{\ell} \to G(\ql) \to  \operatorname{GL}_V(\ql)$. The following lemma is well-known, but we've included a proof for the benefit of the reader.
\begin{Lem} \label{Lem:MonodromyEll}
Let $\shg^{\circ} \subset \shg$ be a connected component and let $a \in \shg^{\circ}(\ovfp)$. Then the Zariski closure of the monodromy representation
\begin{align}
    \pi_1^{\text{\'et}}(\shg^{\circ},a) \to K_{\ell} \to \operatorname{GL}_V(\ql)
\end{align}
corresponding to $R^1 \pi_{\ast} \mathbb{Q}_{\ell}$ is equal to $\gder_{\ql}$.
\end{Lem}
\begin{proof}
It follows from Lemma \ref{Lem:ConnectedComponents} and Lemma \ref{Lem:ConnectedComponentSpecialFibre} that 
\begin{align}
    \pi_0(\shginf) \to \pi_0(\shg)
\end{align}
is a pro-\'etale $\nu(K^p)$-torsor. In particular, the stabiliser of a connected component of $\pi_0(\shginf)$ under the action of $K^p$ is equal to $K^p \cap \gder(\afp)$. If we pass to the induced $K_{\ell}$-torsor
\begin{align}
    \shginfl:=\shginf \times_{K^p} K_{\ell} \to \shg,
\end{align}
then the action of $K_{\ell}$ on a connected component has stabiliser equal to $K_{\ell} \cap \gder(\ql)$. By profinite Galois theory for schemes, cf. \cite[Section 3.1.10, 3.1.11]{OrdinaryHO}, this stabilizer can be identified with the image of the monodromy representation $\pi_1^{\text{\'et}}(\shg^{\circ},a) \to K_{\ell}$. Thus the image is equal to $K_{\ell} \cap \gder(\ql)$, which is a compact open subgroup of $\gder(\ql)$. We thus see that the image has Zariski closure $\gder_{\ql}$ in $\operatorname{GL}_{V, \ql}$.
\end{proof}
To proceed, we will make the following assumption:
\begin{Assump} \label{Assump:1}
If $[b]$ is $\mathbb{Q}$-non-basic, then for any distinguished central leaf $\cby \subset \shgb$ the natural map
\begin{align} \label{eq:LeafIso}
    \cby \to \shg 
\end{align}
induces a bijection on connected components for all $K^p \subset G(\afp)$.
\end{Assump}
\begin{Rem} This assumption holds true if either $G_{\qp}$ splits over an unramified extension or if $\mathbf{Sh}_{K^pK_p}(G,X)$ is proper, by \cite{vH}*{Theorem 4.5.2} (see \cite[Remark 4.3.2]{vH}). More generally, the assumption holds when \cite[Conjecture 4.3.1]{vH} holds for $\shg$. A proof of \cite[Conjecture 4.3.1]{vH} under the assumption that $\gad$ is $\mathbb{Q}$-simple will appear in the forthcoming PhD thesis of Shengkai Mao, see \cite[Corollary 1.6 ]{MaoCompact}. Moreover, the assumption holds unconditionally when $[b]$ is the $\mu$-ordinary element, see Remark \ref{Rem:MuOrdinary}.
\end{Rem}
Let us denote by $J_{b}(\qp)''$ the kernel of the map $J_b(\qp) \to \gab(\qp) \to \pi_1(G)_I^{\sigma}$, note that $J_b''(\qp) \supset J_b'(\qp)$. Moreover, since $\nu(K_p)=\gab(\zp)$ is the kernel of $\gab(\qp) \to \pi_1(G)_I^{\sigma}$ by Proposition \ref{Prop:ImageParahoricII}, the group $J_b''(\qp)$ is just the inverse image of $\gab(\zp) \subset \gab(\qp)$ under the natural map $J_b(\qp) \to \gab(\qp)$.
\begin{Lem} \label{Lem:JbppActs}
If Assumption \ref{Assump:1} holds, then for $y \in \xmub(\ovfp)$ the map
\begin{align} \label{Eq:ProductFixedADLV}
    \pi_{\infty}(-,y):\pi_0(\ig) \times \{y\} \to \pi_0(\shg)
\end{align}
is $J_b(\qp)$-equivariant, where $J_b(\qp)$ acts on $\pi_0(\shg)$ via the inverse of the natural map $J_b(\qp) \to \gab(\qp) \to \pi_1(G)_I^{\sigma}$. In particular, the group $J_b''(\qp)$ acts on the fibers of \eqref{Eq:ProductFixedADLV}.
\end{Lem}
\begin{proof}
We will prove the result for the map
\begin{align}
    \pi_{\infty}(-,y):\pi_0(\iginf) \times \{y\} \to \pi_0(\shginf).
\end{align}
The map only depends on the connected component containing $y$, and thus it suffices to prove the result for one point in each connected component of $\xmub$. Since the map $\pi_{\infty}:\iginf \times \xmub \to \shginf$ is $J_b(\qp)$-invariant, it suffices to prove it for one point in each $J_b(\qp)$-orbit of connected components of $\xmub$. Since $J_b(\qp)$ acts transitively on $\pi_0(\xmub)$ by \cite[Theorem A.1.3]{vH}, it is enough to prove the result for $y=1$. \smallskip 

By Lemma \ref{Lem:UniformisationAndPiNaught} we know that for $g \in J_b(\qp)$ we have
\begin{align}
    \pi_{\infty}(\tilde{x}, g \cdot 1) = \kappa(g) \cdot \pi_{\infty}(\tilde{x}, 1),
\end{align}
where $\kappa(g)$ is the image of $g$ in $\pi_1(G)_I^{\sigma}$. Then the $J_b(\qp)$-invariance of $\pi_{\infty}$ tells us that
\begin{align}
    \pi_{\infty}(g \tilde{x}, g g^{-1}\cdot 1) = \pi_{\infty}(\tilde{x}, g^{-1} \cdot 1) = \kappa(g)^{-1} \pi_{\infty}(\tilde{x}, 1).
\end{align}
Thus the result holds for the connected component of $\iginf$ containing $\tilde{x}$ and therefore for the connected components intersecting the $J_b(\qp)$-orbit of $\tilde{x}$. Since the map $\pi_{\infty}(-,y)$ is $G(\afp)$-equivariant and the $G(\afp)$-action commutes with the $J_b(\qp)$ action on $\iginf$ and with the $\pi_1(G)_I^{\sigma}$ action on $\pi_0(\shginf)$, the result holds for the connected components of the $G(\afp) \times J_b(\qp)$-orbit of $\tilde{x}$. \smallskip 

Assumption \ref{Assump:1} tells us that $H \subset J_b(\qp)$ acts transitively on the fibers of
\begin{align}
    \pi_0(\iginf) \to \pi_0(\shginf).
\end{align}
Now $\pi_1(G)_I^{\sigma} \times G(\afp)$ acts transitively on $\pi_0(\shginf)$ by inspection, see the proof of Lemma \ref{Lem:ConnectedComponents}. We deduce from this that $J_b(\qp) \times G(\afp)$ acts transitively on $\pi_0(\iginf)$, since $H$ is contained in the kernel of $J_b(\qp) \to \pi_1(G)_I^{\sigma}$ by Proposition \ref{Prop:ImageParahoricI}.
\end{proof}
It follows from the proof of Lemma \ref{Lem:JbppActs} that the following result holds.
\begin{Cor} \label{Cor:Transitive}
If Assumption \ref{Assump:1} holds, then the group $G(\afp) \times J_b(\qp)$ acts transitively on $\pi_0(\iginf)$.
\end{Cor}
Let $\Sigma$ be a finite set of primes containing $p$ and all the places $\ell$ such that $\gder_{\ql}$ has a compact factor. Let $\afs$ be the set of finite adeles away from $\Sigma$.
\begin{Lem} \label{Lem:PrimeToPDerivedTrivial}
If Assumption \ref{Assump:1} holds, then the group $\gder(\afs)$ acts trivially on $\pi_0(\iginf)$.
\end{Lem}
\begin{proof}
It follows from Lemma \ref{Lem:ConnectedComponents} and Lemma \ref{Lem:ConnectedComponentSpecialFibre} that $\gder(\afp)$ acts trivially on $\pi_0(\shginf)$. It follows from Assumption \ref{Assump:1} that $\gder(\afp)$ acts trivially on $\pi_0(\cbinf)$ as well, where we recall that we have assumed that $\cb$ is distinguished central leaf. \smallskip

Write $H=\varprojlim H_{n}$ as an inverse limit of finite groups, indexed by the natural numbers. This induces a description of $\iginf \to \cbinf$ as a $G(\afp)$-equivariant inverse limit of finite \'etale covers of $\cbinf$. Concretely, $\iginf=\varprojlim_n \cbinf^n$, where $\cbinf^n$ is the quotient of $\iginf$ by the kernel of $H \to H_{n}$. Since the group $\gder(\afp)$ acts trivially on $\pi_0(\cbinf)$, it has finite orbits when acting on $\pi_0(\cbinf^n)$. In particular, for each $\ell \not \in S$ the group $\gder(\ql)$ acts through a finite quotient on $\pi_0(\cbinf^n)$ for all $n$. 

By the definition of $\Sigma$, the group $\gder_{\ql}$ has no compact factors for $\ell \not \in \Sigma$, which by \cite[Theorem 7.1, Theorem 7.5]{PR} implies that the group $\gder(\ql)$ has no finite index proper subgroups. Thus $\gder(\ql) \subset \gder(\afs)$ acts trivially on $\pi_0(\cbinf^n)$ for all $\ell \not \in \Sigma$, and the result follows by passing to the inverse limit over $n$.
\end{proof}

\subsection{\texorpdfstring{$p$}{p}-adic monodromy} Let $\cbcirc$ be a connected component of $\cb$ and let $z \in \cbcirc(\ovfp)$. Let $\pi_1^{\text{ét}}(\cbcirc,z) \to H$ be the monodromy representation associated to $\ig \to \cb^{\mathrm{perf}}$ and denote its image by $M \subset H \subset J_b(\qp)$. Recall that $J_b' \subset J_b$ is the kernel of the natural map $J_b \to \gab$ as in Section \ref{Sec:NewtonStratification}.

\begin{Thm} \label{Thm:PAdicMonodromyJb}
If Assumption \ref{Assump:1} holds, then the Zariski closure of $M$ in $J_b(\qp)$ is equal to $J_b'$.
\end{Thm}
This result is a consequence of the results of D'Addezio \cites{d2020monodromy, DAddezioII} in combination with Lemma \ref{Lem:MonodromyEll} and Assumption \ref{Assump:1}. To explain this, we first need to introduce some notation. 

\subsubsection{} Recall the following notions from \cite[Sec. 2.2]{d2020monodromy}. Write $F\text{-Isoc}(S)$ for the $\qp$-linear Tannakian category of $F$-isocrystals over a smooth finite type scheme $S$ over $\ovfp$, and write $F\text{-Isoc}^{\dagger}(S)$ for the $\qp$-linear Tannakian category of overconvergent $F$-isocrystals over $S$. There is a natural fully faithful embedding $F\text{-Isoc}^{\dagger}(S) \subset F\text{-Isoc}(S)$, which sends an overconvergent $F$-isocrystal $\mathcal{M}^{\dagger}$ to the underlying $F$-isocrystal $\mathcal{M}$. Similarly, we write $\mathrm{Isoc}^{\dagger}(S)$ and $\mathrm{Isoc}(S)$ for the $\qpbr$-linear category of (overconvergent) isocrystals over $S$. There are natural faithful forgetful functors from (overconvergent) $F$-isocrystals to (overconvergent) isocrystals. 

Given an overconvergent isocrystal $\mathcal{M}^{\dagger}$ over $S$ as above we write $\mathcal{M}$ for its underlying isocrystal. For a point $s \in S(\ovfp)$, there are 
\begin{align}
    \operatorname{Mon}(S, \mathcal{M},s) \subset \operatorname{Mon}(S, \mathcal{M}^{\dagger},s)
\end{align}
that are algebraic groups over $\qpbr$, see the introduction of \cite{DAddezioII}. They are defined to be the Tannakian groups corresponding to the smallest Tannakian subcategory of $\mathrm{Isoc}(S)$ and $\mathrm{Isoc}^{\dagger}(S)$, respectively, containing $\mathcal{M}$, via the fiber functor $\omega_s$
\begin{align}
    \omega_s:\mathrm{Isoc}(S) \to \mathrm{Isoc}(\ovfp)= \operatorname{Vect}_{\qpbr}.
\end{align}
\begin{align} \label{Eq:FCrystalGStructure}
    \operatorname{Rel}_p:\operatorname{Rep}_{\qp} G \to F\text{-Isoc}(\shg)
\end{align}
such that the representation $G_{\qp} \to \mathcal{G}_V \to \operatorname{GL}_V$ coming from the choice of Hodge embedding is sent to the $F$-isocrystal $\mathcal{M}$. Since $\mathcal{M}$ is an overconvergent $F$-isocrystal, it follows that this tensor functor factors through an exact $\qp$-linear tensor functor
\begin{align} \label{Eq:FCrystalGStructureOC}
    \operatorname{Rel}_p:\operatorname{Rep}_{\qp} G \to F\text{-Isoc}^{\dagger}(\shg),
\end{align}
see \cite[Lemma 3.3.2]{OrdinaryHO}. 

Choose an isomorphism $\mathbb{D}_z \simeq \mathbb{D}_x$ sending $s_{\alpha,\mathrm{cris},z}$ to $s_{\alpha,\mathrm{cris},x}$ for all $\alpha \in \mathscr{A}$. If we compose this with our fixed isomorphism $\zpbr \otimes V_p \to \mathbb{D}_x$, we get an isomorphism $\zpbr \otimes V_p \to \mathbb{D}_z$ which sends $1 \otimes s_{\alpha}$ to $s_{\alpha,\mathrm{cris},z}$ for all $\alpha \in \mathscr{A}$. This induces an isomorphism $\omega_z(\mathcal{M}^{\dagger})=\mathbb{D}_z[1/p] \to V \otimes \qpbr$ sending $\omega_z(s_{\alpha})=s_{\alpha,\mathrm{cris},z}$ to $s_{\alpha} \otimes 1$ for all $\alpha \in \mathscr{A}$. This identifies the composite $\omega_z \circ \operatorname{Rel}_p:\operatorname{Rep}_{\qp} G \to \operatorname{Vect}_{\qpbr}$ with the standard fiber functor, tensored up to $\qpbr$. Thus if we apply Tannakian duality to \eqref{Eq:FCrystalGStructure} and \eqref{Eq:FCrystalGStructureOC}, we get inclusions 
\begin{align}
    \operatorname{Mon}(Z, \mathcal{M},z) \subset \operatorname{Mon}(Z, \mathcal{M}^{\dagger},z) \subset G \otimes \qpbr \subset \operatorname{GL}(V \otimes \qpbr).
\end{align}
\begin{Lem} \label{Lem:OCPAdicMonodromy}
If Assumption \ref{Assump:1} holds, then the monodromy group $$\operatorname{Mon}(\cbcirc, \mathcal{M}^{\dag}) \subset G \otimes \qpbr$$ is equal to $\gder \otimes \qpbr$.
\end{Lem}
\begin{proof}
The geometric monodromy group of $R^1 \pi_{\ast} \ql$ over $\cbcirc$ is isomorphic to $\gder_{\ql}$, by Lemma \ref{Lem:MonodromyEll} and Assumption \ref{Assump:1}. Then \cite{d2020monodromy}*{Theorem 1.2.1} tells us that $\mathbb{M}_p:=\operatorname{Mon}(\cbcirc, \mathcal{M}^{\dag})$ is isomorphic to $\gder$ over an algebraic closure of $\qpbr$. This implies that $\mathbb{M}_p$ is equal to its own derived subgroup and therefore it is contained in $\gder \otimes \qpbr$. Since $\mathbb{M}_p$ is connected and of the same dimension as $\gder$, it follows that the inclusion $\mathbb{M}_p \subset \gder \otimes \qpbr$ is an equality.
\end{proof}
Since $\cbcirc$ is contained in a single Newton stratum, the $F$-isocrystal $\mathcal{M}$ admits a unique slope filtration $S_{\bullet}(\mathcal{M})$. It is explained in \cite[Lemma 3.3.4 and the paragraph preceding it]{OrdinaryHO} that this gives rise to a fractional cocharacter $\lambda$ of $G \otimes \qpbr$. Let $b_z \in G(\qpbr) \subset \operatorname{GL}(V \otimes \qpbr)$ be the element corresponding to the Frobenius on $\omega_z(\mathcal{M}^{\dagger})=\mathbb{D}_z[1/p]=V \otimes \qpbr$.
Then by construction of our identification $\mathbb{D}_z[1/p]=V \otimes \qpbr$ we have $b_z=b$. It follows from \cite[Lemma 3.3.4]{OrdinaryHO} that $\lambda=\nu_b$, where $\nu_b$ is the Newton cocharacter attached to $b$, see \cite[Section 1.1.2]{KMPS}. As explained in \cite[Sec. 4.1]{DAddezioII}, we find that the monodromy group
\begin{align}
    \operatorname{Mon}(\cbcirc, \mathcal{M}) \subset G_{\qpbr}
\end{align}
is contained in the parabolic subgroup $P(\lambda) \subset G_{\qpbr}$ associated to $\lambda$.
\begin{Lem} \label{Lem:PAdicMonodromy}
If Assumption \ref{Assump:1} holds, then the monodromy group $$\operatorname{Mon}(\cbcirc, \mathcal{M}) \subset \operatorname{Mon}(\cbcirc, \mathcal{M}^{\dag})=\gder \otimes \qpbr$$ is equal to the intersection of $P(\lambda)$ with $\gder \otimes \qpbr$.
\end{Lem}
\begin{proof}
This is \cite[Theorem 5.1.2]{DAddezioII}.
\end{proof}
We also consider the centraliser $M(\lambda) \subset P(\lambda)$ of $\lambda$. 
\begin{Lem}\label{Lem:SemiSimplePAdicMonodromy}
Let $\mathcal{N}=\operatorname{Gr}S_{\bullet}(\mathcal{M})$ be the associated graded of the slope filtration $S_{\bullet}(\mathcal{M})$ on $\mathcal{M}$. Then the inclusion
\begin{align}
    \operatorname{Mon}(\cbcirc, \mathcal{N}) \subset \operatorname{Mon}(\cbcirc, \mathcal{M})
\end{align}
identifies $\operatorname{Mon}(\cbcirc, \mathcal{N})$ with the intersection of $\gder \otimes \qpbr$ and $M(\lambda)$ in $G \otimes \qpbr$.
\end{Lem}
\begin{proof}
This is \cite[Proposition 5.1.4]{DAddezioII}.
\end{proof}

\subsection{The proof of Theorem \ref{Thm:PAdicMonodromyJb}} \label{Sec:ProofOfMonodromy} We will deduce Theorem \ref{Thm:PAdicMonodromyJb} from the results proved above in combination with Proposition \ref{Prop:FactIsocrystals} below. \smallskip 

Let $S$ be a smooth connected scheme over $\ovfp$ and let $\pi:A \to S$ be an abelian scheme such that the $p$-divisible group $X=A[p^{\infty}]$ is completely slope divisible; let $S_{\bullet}(X)$ be the slope filtration of $X$. Let $\mathcal{M}$ be the isocrystal attached to $A$ and let $S_{\bullet} \mathcal{M}$ be the slope filtration of $\mathcal{M}$. Let $s \in S(\ovfp)$, let $\mathbb{X}=X_s$ and let $\tilde{S} \to S$ be the scheme representing the functor sending an $S$-scheme $T$ to the set $\operatorname{Isom}(\operatorname{Gr} S_{\bullet}(X_T), \mathbb{X}_T)$; it is a pro-\'etale torsor for the profinite group $\operatorname{Aut}(\mathbb{X})$ by \cite[Corollary 1.10]{OortZink}. The rational Dieudonn\'e module functor gives a natural continuous homomorphism
\begin{align}
    \operatorname{Aut}(\mathbb{X}) \to \operatorname{Aut}_{\qpbr}(\mathcal{M}_s).
\end{align}
\begin{Prop} \label{Prop:FactIsocrystals}

Let $\rho:\pi_{1}^{\text{\'et}}(S,s) \to \operatorname{Aut}(\mathbb{X})$ be the monodromy representation associated to $\tilde{S} \to S$. Then the Zariski closure of the image of $\rho$ inside $\operatorname{Aut}_{\qpbr}(\mathcal{M}_s)$ is equal to the monodromy group
\begin{align}
    \mon(S,\mathcal{N},s) \subset \operatorname{GL}_{\qpbr}(\mathcal{M}_s),
\end{align}
where $\mathcal{N} = \operatorname{Gr}S_{\bullet}(\mathcal{M})$.
\end{Prop}
\begin{proof}
As explained in the proof of \cite[Theorem 5.16]{DAddezioII}, the $F$-isocrystal $\mathcal{M}$ is the rational crystalline Dieudonn\'e crystal $\mathbb{D}(X)[1/p]$ of $X$ and the $F$-isocrystal $\mathcal{N}$ is therefore the crystalline Dieudonn\'e crystal of $\operatorname{Gr} S_{\bullet}(X)$. Thus the tautological isomorphism
\begin{align}
    \operatorname{Gr} S_{\bullet}(X_{\tilde{S}}) \to \mathbb{X}_{\tilde{S}}
\end{align}
induces an isomorphism
\begin{align} \label{Eq:TrivialisationCanonical}
    \mathcal{N}_{\tilde{S}} \to \mathcal{M}_{s, \tilde{S}}.
\end{align}
Let $Z$ be a connected component of $\tilde{S}$. Then the stabiliser in $\operatorname{Aut}(\mathbb{X})$ of $Z$ is equal to $M$, the image of the monodromy representation. Moreover, $Z \to S$ is a pro-\'etale $M$-torsor. 

Now let $\eta_S=\spec k(\eta_S)$ be the generic point of $S$ and let $\eta_Z=\spec (k(\eta_Z))$ be the generic point of $Z$. Then $\eta_Z \to \eta_S$ is a pro-\'etale $M$-torsor, in other words, the field $k_{\eta_Z}$ is a Galois extension of $k(\eta_S)$ with Galois group topologically isomorphic to $M$. Let $\mathcal{N}_{\eta_S}$ be the pullback of the isocrystal $\mathcal{N}$ to $\eta_S$, and let $\langle \mathcal{N}_{\eta_S} \rangle$ be the Tannakian category generated by $\mathcal{N}_{\eta_S}$ inside the Tannakian category of isocrystals on $\eta_S$. (The field $k(\eta_S)$ has a finite $p$-basis, and so the category of isocrystals on it is Tannakian by \cite[Corollary 3.3.3]{DrinfeldCrystals}.) By \cite[Theorem 3.2.5]{DAddeziovanHoften} and its proof, pullback to $\eta_S$ induces an equivalence of Tannakian categories
\begin{align}
    \langle \mathcal{N} \rangle \simeq \langle \mathcal{N}_{\eta_S} \rangle.
\end{align}
We can restrict \eqref{Eq:TrivialisationCanonical} to $\eta_Z$ to deduce that the isocrystal $\mathcal{N}_{\eta_S}$ becomes trivial after pullback to $\eta_Z$. Since $a:Z \to S$ is a pro-\'etale $M$-torsor, it satisfies descent for isocrystals by \cite[Proposition 3.5.4]{DrinfeldCrystals}, see also \cite[main result]{MatthewDescent} or \cite[Section 2]{BhattScholzePrismaticCrystals}. By descent for isocrystals, any object $P$ in $\langle \mathcal{N}_{\eta_S} \rangle$ can be described by its pullback $a^{\ast} P$, which is a $\qpbr$-vector space, together with its continuous and $\qpbr$-linear action of $M$. We identify $\mathcal{N}_{\eta_S}$ itself with the vector space $\omega_s(\mathcal{N})=\omega_s(\mathcal{M})=:\mathcal{M}_s$.

This identifies $\langle \mathcal{N}_{\eta_S} \rangle$ with a full subcategory of the category of continuous representations of $M$ on $\qpbr$-vector spaces. Namely, the one generated (as a Tannakian category) by the representation $M \to \operatorname{Aut}(\mathbb{X}) \to \operatorname{GL}_{\qpbr}(\mathcal{M}_s)$. \smallskip 

This category is also equivalent to the full subcategory $\langle \mathcal{M}_s \rangle$ of the category of (not necessarily continuous) representations of $M$ on finite-dimensional $\qpbr$-vector spaces, generated (as a Tannakian category) by the representation $M \to \operatorname{Aut}(\mathbb{X}) \to \operatorname{GL}_{\qpbr}(\mathcal{M}_s)$. Indeed, the continuity of the action of $M$ is automatic for objects in $\langle \mathcal{M}_s \rangle$. By \cite[Proposition 6.5.15]{FundamentalGroups}, this implies that the Tannakian group of $\langle \mathcal{M}_s \rangle$ is equal to the Zariski closure of $M$ in $\operatorname{GL}_{\qpbr}(\mathcal{M}_s)$. But this Tannakian group is equal to $\operatorname{Mon}(S, \mathcal{N},s)$ per definition.
\end{proof}

\begin{proof}[Proof of Theorem \ref{Thm:PAdicMonodromyJb}]
Recall from \cite[Section 1.1.4]{KMPS} that there is a natural map $J_b \otimes \qpbr \to G \otimes \qpbr$ whose image can be identified with the centraliser $M_{\nu_b}$ of $\nu_b$. As explained in the paragraph after Lemma \ref{Lem:OCPAdicMonodromy}, the Newton cocharacter $\nu_{b}$ is equal to $\lambda$. Thus, by Lemma \ref{Lem:SemiSimplePAdicMonodromy}, the monodromy group $\operatorname{Mon}(\cbcirc, \mathcal{N})$ is the intersection of $M_{\nu_b} \otimes \qpbr$ and $\gder \otimes \qpbr$.

\textbf{Step 1: The Mantovan Igusa variety.} Let $\cbgsp$ be a central leaf in $\shgsp$ containing $\cb$ and let $X=A[p^{\infty}]$ be the $p$-divisible group of the universal abelian variety over $\cbgsp$. Then because we have chosen $\cb$ to be distinguished, it follows that the $p$-divisible group $\mathbb{X}:=A_x[p^{\infty}]$, where $x$ is our fixed basepoint, is completely slope divisible in the sense of \cite[Definition 2.4.2]{KimLeaves}, see \cite[Lemma 2.4.3, paragraph after Definition 2.4.2]{KimLeaves}. As explained in \cite[Section 3.2.3]{MantovanPEL}, this implies that $X$ is completely slope divisible over $\cbgsp$. \smallskip

Let $S_{\bullet}(X)$ be the slope filtration of $X$, and let $\operatorname{Gr} S_{\bullet}(X)$ be the associated graded. Let $\lambda$ (not to be confused with the slope cocharacter $\lambda$ introduced above) be the polarisation of $\mathbb{X}$ induced by the polarisation on $A_x$ and let $\operatorname{Aut}(\mathbb{X}, \lambda)$ be the profinite group of automorphisms of $\mathbb{X}$ that preserve $\lambda$ up to a scalar in $\zp^{\times}$. Then it follows from \cite[main result]{MantovanPEL}, see \cite[discussion after Definition 4.3.6]{CaraianiScholze}, that there is a pro-\'etale $\operatorname{Aut}(\mathbb{X}, \lambda)$-torsor
\begin{align}
    \pi:\igmgsp \to \cbgsp
\end{align}
parametrising isomorphisms $\operatorname{Gr} S_{\bullet}(X_{\cb}) \simeq \mathbb{X}_{\cbgsp}$ compatible with the polarisations up to a scalar in $\zp^{\times}$. By \cite[Proposition 4.3.8]{CaraianiScholze}, the perfection of $\pi$ can be identified with
\begin{align}
    \iggsp \to \cbgsp^{\mathrm{perf}}. 
\end{align}
There is moreover a commutative diagram
\begin{equation}
    \begin{tikzcd}
    \igm \arrow{d}{a} \arrow{r} & \igmgsp \arrow{d} \\
    \cb \arrow{r} & \cbgsp,
    \end{tikzcd}
\end{equation}
where $\igm$ is the pro-\'etale $H$-torsor associated to $\ig \to \cb^{\mathrm{perf}}$, under the equivalence of \'etale sites between $\cb$ and $\cb^{\mathrm{perf}}$. Note that in the Siegel case, the Mantovan Igusa variety depends on $\llbracket b \rrbracket$ rather than just on $[b]$. Therefore we have decided to include $\llbracket b \rrbracket$ rather than $[b]$ in the notation for the Mantovan Igusa varieties. Our fixed isomorphism $\zpbr \otimes V_p \to \mathbb{D}_x$ sending $1 \otimes s_{\alpha}$ to $s_{\alpha,\mathrm{cris},x}$ for all $\alpha \in \mathscr{A}$ induces a natural embedding
\begin{align}
    J_b(\qp) \to \operatorname{QIsog}_{\lambda}(\mathbb{X}),
\end{align}
which maps $H \subset J_b(\qp)$ to $\operatorname{Aut}(\mathbb{X}, \lambda)$. The morphism $\igm \to \igmgsp$ is $H$-equivariant for the $H$-action on the target via $H \to \operatorname{Aut}(\mathbb{X}, \lambda)$. In particular, $\igm \to \cb$ naturally maps $H$-equivariantly to the pro-\'etale $\operatorname{Aut}(\mathbb{X})$-torsor over $\cb$ that was introduced in the beginning of Section \ref{Sec:ProofOfMonodromy}.

\textbf{Step 2: Applying Proposition \ref{Prop:FactIsocrystals}} It now follows from Proposition \ref{Prop:FactIsocrystals} that the Zariski closure of the image of the monodromy representation $\rho:\pi_1^{\text{\'et}}(\cbcirc, x) \to H \to \operatorname{Aut}(\mathbb{X}, \lambda) \to \operatorname{GL}_V(\qpbr)$ is equal to the monodromy group $\mon(\cbcirc, \mathcal{N},x)$. This monodromy group is equal to $\left(M_{\nu_b} \cap \gder \right) \otimes \qpbr$ by Lemma \ref{Lem:SemiSimplePAdicMonodromy}. Thus the Zariski closure of $M$ in $G(\qpbr)$ is equal to $\left(M_{\nu_b} \cap \gder \right) \otimes \qpbr = J'_{b} \otimes \qpbr$. \smallskip 

We conclude that the image of $M \to J_b(\qp) \to J_b(\qpbr)$ is contained in $J_b'(\qpbr)$ and thus in $J_b'(\qp)$. Moreover, this image is Zariski dense in $J_b'(\qpbr)$ and therefore so in $J_b'(\qp)$, since the formation of Zariski closures commutes with flat base change.
\end{proof}
\subsection{Consequences of Theorem \ref{Thm:PAdicMonodromyJb}} In this section we will deduce some consequences of Theorem \ref{Thm:PAdicMonodromyJb} that are relevant to us.
\begin{Cor} \label{Cor:PAdicMonodromy}
If Assumption \ref{Assump:1} holds, then the group $M$ contains a compact open subgroup of $\jbder(\qp)$ and is contained in $J_b'(\qp)$.
\end{Cor}
\begin{proof}
The group $M$ is a $p$-adic Lie group by \cite[Prop. 2.3]{LecturesLieGroups} and the morphism $M \to H \to J_b(\qp)$ is a morphism of $p$-adic Lie groups by \cite[Prop. 2.2]{LecturesLieGroups}. This implies that there is a $\qp$-Lie algebra $\operatorname{Lie} M$ and a morphism of Lie algebras $\operatorname{Lie} M \to \operatorname{Lie} J_b(\qp) = \operatorname{Lie} J_b$. By Theorem \ref{Thm:PAdicMonodromyJb}, the group $M$ has Zariski closure equal to $J_b'$ and is thus contained in $J_b'(\qp)$.

This means that $\operatorname{Lie} J_b'$ is the smallest Lie algebra of an algebraic subgroup of $J_b$ containing $\operatorname{Lie} M$. In the notation of \cite[Section 7.1]{Borel}, this is expressed as $\mathfrak{a}(\operatorname{Lie} M)=\operatorname{Lie} J_b'$. By \cite[Corollary 7.9]{Borel} we have the following equality of Lie subalgebras of $\operatorname{Lie} J_b$: 
\begin{align}
    [\operatorname{Lie} M,\operatorname{Lie} M] &= [\mathfrak{a}(\operatorname{Lie} M),\mathfrak{a}(\operatorname{Lie} M)] \\
    &=[\operatorname{Lie} J_b',\operatorname{Lie} J_b']\\
    &=\operatorname{Lie} \jbder.
\end{align} 
In particular, we see that $\operatorname{Lie} \jbder \subset \operatorname{Lie} M$. By the theory of $p$-adic Lie groups and their exponential maps, see \cite[Section 2]{LecturesLieGroups}, this implies that $M$ contains a compact open subgroup of $\jbder(\qp)$.
\end{proof}
Let $Z$ be the center of the algebraic group $J_b$.
\begin{Lem} \label{Lem:CenterFinitelyManyOrbits}
If Assumption \ref{Assump:1} holds, then the quotient $Z(\qp) \backslash \pi_0(\ig)$ is finite.
\end{Lem}
\begin{proof}
The quotient $H \backslash \pi_0(\ig)$ can be identified with $\pi_0(\cb)$ and is therefore finite. Moreover, the stabiliser of a connected component of $\pi_0(\ig)$ can be identified with $M \subset H$. Since $M$ contains a compact open subgroup of $\jbder(\qp)$ by Corollary \ref{Cor:PAdicMonodromy}, we see that $M \cap \jbder(\qp)$ is of finite index in $H^{\mathrm{der}}:=H \cap \jbder(\qp)$. Thus the map $\pi_0(\ig) \to H^{\mathrm{der}} \backslash \pi_0(\ig)$ is finite-to-one, and $H$ acts with finitely many orbits on the target and its action moreover factors through $H/H^{\mathrm{der}}$.

The quotient $H/H^{\mathrm{der}}$ can be identified with a compact open subgroup of $J_b^{\mathrm{ab}}(\qp)$. In particular, a compact open subgroup of $Z(\qp)$ will act on it with finitely many orbits, and the lemma follows.
\end{proof}
\begin{Cor} \label{Cor:FiniteOrbits}
If Assumption \ref{Assump:1} holds, then the group $\jbder(\qp)$ acts with finite orbits on $\pi_0(\ig)$.
\end{Cor}
\begin{proof}
Let $a \in \pi_0(\ig)$ and write the $J_b(\qp)$-orbit of $a$ as $J_b(\qp)/P_a$, where $P_a \subset J_b(\qp)$ is the stabiliser of $a$. We want to show that the group $P_a^{\mathrm{der}}:=P_a \cap \jbder(\qp)$ has finite index in $\jbder(\qp)$. Equivalently, by the fact that $\jbder(\qp) \cdot Z(\qp)$ has finite index in $J_b(\qp)$ and the fact that $Z(\qp) \cap \jbder(\qp)$ is finite, we need to show that $P_a^{\mathrm{der}} \cdot Z(\qp)$ has finite index in $J_b(\qp)$. \smallskip 

Lemma \ref{Lem:CenterFinitelyManyOrbits} implies that $P_a \cdot Z(\qp)$ has finite index in $J_b(\qp)$, and so it suffices to show that 
\begin{align}
    P_a^{\mathrm{der}} \cdot Z(\qp) \subset P_a \cdot Z(\qp)
\end{align}
has finite index. This is true because the cokernel of $P_a^{\mathrm{der}} \to P_a$ is naturally a subgroup of $J_b^{\mathrm{ab}}(\qp)$, and because $Z(\qp) \to J_b^{\mathrm{ab}}(\qp)$ has finite cokernel. 
\end{proof}
The group $\jbder$ is simply connected because $J_b$ is an inner form of a Levi subgroup of $G$, and $\gder$ is simply connected (see \cite{MalleTesterman}*{Proposition 12.14}). Therefore we can write $\jbder$ as a product of restrictions of scalars of semi-simple and simply connected groups whose adjoint groups are absolutely simple. In particular, we can write $\jbder=\jbderan\times \jbderiso$ with the first factor anisotropic and the second factor totally isotropic. In other words, the group $\jbderiso$ has no compact factors.
\begin{Prop} If Assumption \ref{Assump:1} holds, then $\jbderiso(\qp)$ acts trivially on $\pi_0(\ig)$. \label{Prop:Geometricmonodromy}
\end{Prop}

\begin{proof}
It follows from Corollary \ref{Cor:FiniteOrbits} that $\jbder(\qp)$ acts with finite orbits on $\pi_0(\ig)$. Therefore the subgroup $\jbderiso(\qp)$ acts with finite orbits. However, since $\jbderiso$ is totally isotropic, it follows that $\jbderiso(\qp)$ has no finite index proper subgroups, see \cite[Theorem 7.1, Theorem 7.5]{PR}. We conclude that the action of $\jbderiso(\qp)$ on $\pi_0(\ig)$ is trivial. \end{proof}
This argument will \emph{not} work for the anisotropic part of $\jbder$, because it is not true that $\jbderan(\qp)$ has no non-trivial finite quotients. This is why we have to assume that $\jbder=\jbderiso$ in our main theorems.

\section{Constructing a maximal torus} \label{Section4} In this section, we will show that there are (many) maximal tori of $J_b'$ whose $\qp$-points stabilise given connected components of $\ig$.
\subsection{Prime-to-\texorpdfstring{$p$}{p} Hecke operators}
Let the notation be as in Section \ref{Sec:Notation} and let $\tilde{z} \in \iginf(\ovfp)$ be a point with image $z \in \cbinf(\ovfp)$. By the construction of the Igusa variety $\ig \to \cb$ in Section \ref{Sec:Igusa}, we know that the point $\tilde{z}$ corresponds to an isomorphism $A_z[p^{\infty}] \to \mathbb{X}=A_x[p^{\infty}]$, which induces an isomorphism $\mathbb{D}_z \simeq \mathbb{D}_x$ sending $s_{\alpha,\mathrm{cris},z}$ to $s_{\alpha,\mathrm{cris},x}$ for all $\alpha \in \mathscr{A}$. If we compose this with our fixed isomorphism $\zpbr \otimes V_p \to \mathbb{D}_x$, which sends $1 \otimes s_{\alpha}$ to $s_{\alpha,\mathrm{cris},x}$ for all $\alpha \in \mathscr{A}$, then we get an induced embedding $j_{\tilde{z},p}:I_{z,\qp} \to J_b$. Moreover, the image of $\tilde{z}$ in $\shginf(\ovfp)$ gives us an embedding
\begin{align}
    j_{\tilde{z}}^p:I_{z,\afp} \to G_{\afp}.
\end{align}
We note that both $G(\afp)$ and $J_b(\qp)$ act on $\iginf$.
\begin{Lem} \label{Lem:TwoActions}
The subgroup $(j_{\tilde{z}}^p, j_{\tilde{z},p})(I_{z}(\mathbb{Q})) \subset G(\afp) \times J_b(\qp)$, stabilises the point $\tilde{z} \in \iginf(\ovfp)$.
\end{Lem}
\begin{proof}
It suffices to show this in the Siegel case, where it is a direct consequence of the moduli interpretation of the Caraiani--Scholze Igusa variety with infinite prime-to-$p$ level
\begin{align}
    \varprojlim_{\mathcal{K}^p \subset \mathcal{G}_V(\afp)} \iggsp,
\end{align}
coming from \cite[Lemma 4.3.4]{CaraianiScholze}. 
\end{proof}
By \cite[Lemma 2.2.8]{KMPS}, there is a homomorphism $I_z \to \gab$ such that the induced morphism $I_{z, \afp} \to \gab_{\afp}$ agrees with the composition of $j_{\tilde{z}}^p$ with $G_{\afp} \to \gab_{\afp}$, and such that the induced morphism $I_{z,\qp} \to \gab_{\qp}$ agrees with the composition of $j_{\tilde{z},p}$ with $J_{b} \to \gab_{\qp}$. We define $I_{z}' \subset I_z$ to be the kernel of this homomorphism.
\begin{Cor} \label{Cor:PrimeToPHecke}
If Assumption \ref{Assump:1} holds, then the group $j_{\tilde{z},p}\left(I_{z}'(\mathbb{Q})\right) \subset J_b(\qp)$ acts trivially on the image of $\tilde{z}$ in $\pi_0(\iginf)$.
\end{Cor}
\begin{proof}
By Lemma \ref{Lem:TwoActions}, the subgroup
\begin{align}
    (j_{\tilde{z}}^p, j_{\tilde{z},p})(I_{z}(\mathbb{Q})) \subset G(\afp) \times J_b(\qp)
\end{align}
stabilises $\tilde{z}$. The group $\gder(\afs)$ acts trivially on $\pi_0(\iginf)$ by Lemma \ref{Lem:PrimeToPDerivedTrivial} and thus stabilises the image of $\tilde{z}$ in $\pi_0(\iginf)$. Now $I_z(\mathbb{Q}) \cap \gder(\afs) = I_z'(\mathbb{Q})$ by the discussion in the paragraph before the statement of Corollary \ref{Cor:PrimeToPHecke}. If we combine this with Lemma \ref{Lem:TwoActions}, we see that $j_{\tilde{z},p}(I_{z}'(\mathbb{Q}))\subset J_b(\qp)$ also stabilises the image of $\tilde{z}$ in $\pi_0(\iginf)$.
\end{proof}
\begin{Cor} \label{Cor:ClosureActsTrivially}
The closure of $j_{\tilde{z},p}\left(I_{z}'(\mathbb{Q})\right) \subset J_b(\qp)$ acts trivially on the image of $\tilde{z}$ in $\pi_0(\iginf)$ and the image of $\tilde{z}$ in $\pi_0(\ig)$.
\end{Cor}
\begin{proof}
Recall that for a quasicompact and quasiseparated scheme $X$ the topological space $\pi_0(X)$ of connected components of $X$ is a profinite set, see \cite[Lemma 0906]{stacks-project}. This applies in particular to $\ig$ and $\iginf$, and thus $\pi_0(\ig)$ and $\pi_0(\iginf)$ are compact Hausdorff topological spaces. \smallskip

Since the action of $J_b(\qp)$ on $\pi_0(\iginf)$ is continuous by Lemma \ref{Lem:Continuity}, it follows that the stabiliser of the image of $\tilde{z}$ in $\pi_0(\iginf)$ under the action of $J_b(\qp)$ is a closed subgroup. Since the stabiliser contains $I_{z}'(\mathbb{Q}) \subset J_b(\qp)$ by Corollary \ref{Cor:PrimeToPHecke}, it thus contains its closure. The result for $\pi_0(\ig)$ follows from the $J_b(\qp)$-equivariance of $\pi_0(\iginf) \to \pi_0(\ig)$.
\end{proof}

\subsection{Hypersymmetric points}
We start with the following definition (compare with \cite[Def. 6.4]{chai2006hypersymmetric}).
\begin{Def}
We call $z \in \shg(\ovfp)$ \emph{weakly hypersymmetric} if $I_{z, \qp} \simeq J_b$. If in addition $I_z(\mathbb{Q})$ is dense in $I_{z}(\qp)$, then we call $z$ \emph{hypersymmetric}.
\end{Def}
If the Newton stratum $\shgb$ contains a hypersymmetric point, then the arguments above can be used to show that $J_b'(\qp)$ acts trivially on $\operatorname{Ig}_b$, without using the results of Section \ref{Sec:Monodromy}. Unfortunately, although Newton strata on Siegel modular varieties always contain hypersymmetric points, see \cite{chai2006hypersymmetric}, they are sparse in more general settings. For example, in the PEL case, not every Newton stratum contains a hypersymmetric point, see \cite{zong2008hypersymmetric}; they might not exist even in the $\mu$-ordinary stratum, see \cite{LXXiao}*{Corollary 7.5.}. See also \cite[Theorem 1.0.1]{Fung} for a precise criterion for the existence of hypersymmetric points in the Hodge type case, and \cite[Chapter 5, 5.1.1]{Fung} for computations and counterexamples to the existence of hypersymmetric points.
\subsection{Honda--Tate theory} \label{Sec:HondaTate}
Recall that a special point datum $(T,h,i)$ for $(G,X)$ consists of a Shimura datum $(T,h)$, where $T$ is a maximal torus of $G$, and an embedding of Shimura data $i:(T,h) \to (G,X)$. The $\qpbar$-points of $\mathbf{Sh}_{K^p}(G,X)$ that lie in the Shimura variety for $(T,h)$ are called special points; they all have good reduction and so give rise to $\ovfp$-points $z$ of $\shginf$ called special points. These points $z$ come equipped with morphisms $j:T \to I_{z}$ such that $T_{\ql} \to I_{z,\ql} \to G_{\ql}$ is conjugate to $i_{\ql}$ for all $\ell \not=p$. (When $K_p$ is hyperspecial, this is explained in the proof of \cite[Proposition 5.7.6.(ii)]{KisinShinZhu}, and the same proof works in general.) In particular, the natural morphism $T \to I_{z} \to \gab$ is equal to $T \to G \to \gab$. The main theorem of \cites{KMPS} shows that every Newton stratum in $\shg$ contains many special  points. For a maximal torus $\mathscr{T} \subset I_z$, we will write $\mathscr{T}'$ for $\mathscr{T} \cap I_z'$.
\begin{Prop} \label{Prop:HondaTate}
If Assumption \ref{Assump:1} holds, then for each maximal torus $T \subset J_b$ and every connected component $W$ of $\cbinf$, there is a point $z$ in $W$ such that $I_{z}$ contains a torus $\mathscr{T}$, with $\mathscr{T}'_{\qp}$ isomorphic to $T \cap J_b'$. Furthermore, we can choose $\mathscr{T}$ such that $\mathscr{T}'$ satisfies weak approximation, that is, such that $\mathscr{T}'(\mathbb{Q})$ is dense in $\mathscr{T}'(\qp)$. 
\end{Prop}
\begin{proof}
Since $G_{\qp}$ is quasi-split, the group $M_b \subset G_{\qp}$ is quasi-split, and thus we can transfer $T$ from $J_b$ to $M_b$ and hence consider it as a maximal torus of $G_{\qp}$. This transfer can be done such that $T \cap \gder_{\qp}$ is isomorphic to $T \cap J_b'$. Indeed, the same argument applies to show that $T \cap J_b'$ transfers to $M_b \cap \gder$, and $T \cap J_b'$ extends uniquely to $J_b$ and its transfer to $M_b \cap \gder$ extends uniquely to $M_b \cap G$. Then \cite{KMPS}*{Proposition 1.2.5} tells us that we can find a maximal torus $i:\mathscr{T} \to G$ such that: There exists $h \in X$ that factors through $i(\mathscr{T}_{\mathbb{R}})$ making $(\mathscr{T},h,i)$ into a special point datum, and this special point datum induces an isogeny class $\mathscr{I} \subset \shgbinf(\ovfp)$ with automorphism group $I$ containing $\mathscr{T}$, such that $\mathscr{T}_{\qp}$ is $G(\qp)$-conjugate to $T$ in $G_{\qp}$. In particular, $\mathscr{T}_{\qp} \cap \gder_{\qp}$ is isomorphic to $T \cap \gder_{\qp}$. Since $\mathscr{T} \to I_z \to \gab$ is equal to $\mathscr{T} \to G \to \gab$, this implies that $\mathscr{T}'_{\qp}$ is isomorphic to $T \cap \gder_{\qp}$.

The construction of $\mathscr{T}$ in the proof of \cite{KMPS}*{Proposition 1.2.5} is quite flexible. They start by choosing a maximal torus $T_{\infty} \subset G_{\mathbb{R}}$ such that there is an $h \in X$ factoring through $\mathscr{T}_{\infty}$ and then they choose \emph{any} maximal torus $\mathscr{T}_0 \subset G$ that is $G(\mathbb{R})$-conjugate to $T_{\infty}$ and $G(\qp)$-conjugate to $T$. Next, they choose $g \in \gder(\overline{\mathbb{Q}})$ so that the cocycle $\sigma \mapsto g \sigma(g)^{-1}$ lies in $\mathscr{T}_{0}'(\overline{\mathbb{Q}})$, where $\mathscr{T}'_0=\mathscr{T}_0 \cap \gder$. By \cite[Lemma 1.2.1]{KMPS}, this cocycle can be chosen such that its cohomology class is trivial in $H^1(\qp, \mathscr{T}_{0}')$, and in fact the proof of the lemma shows that we can choose this cocycle to be trivial at any finite set of places $S$ of $\mathbb{Q}$ disjoint from $\{\infty\}$. Then $\mathscr{T}$ arises as $\mathrm{int}(g^{-1})(\mathscr{T}_{0,\overline{\mathbb{Q}}})$, which is defined over $\mathbb{Q}$. \medskip

By weak approximation for the variety of maximal tori of $G$, we can choose $\mathscr{T}_0$ as above with fixed $\gder(\ql)$-conjugacy class for any finite set $S$ of primes $\ell \not=p$. By the discussion above, we can choose $g \in \gder(\overline{\mathbb{Q}})$ such that $\mathscr{T}_{0}'$ has the same specified $\gder(\ql)$-conjugacy class for $\ell \in S$. It now follows from \cite{PrasadRapinchuk}*{Theorem 1.(i)} and the result of Klyachko mentioned afterwards (see \cite[part (2) of the Theorem at the end of Section 8.3]{VoskresenskiII}), that this gives us enough flexibility to choose $\mathscr{T} \cap \gder$ to have weak approximation. \smallskip

The proposition now follows from the fact that the isogeny class $\mathscr{I}$ surjects onto $\pi_0(\shginf)$ and moreover intersects every central leaf in the Newton stratum $\shgb$; this follows from Rapoport--Zink uniformisation of isogeny classes (which is \cite[Proposition 6.5]{Zhou}). Indeed, if we let $b$ correspond to $\tilde{z} \in \iginf(\ovfp)$ above $z \in \mathscr{I}$ then the isogeny class $\mathscr{I}$ receives the Rapoport--Zink uniformisation map $\Theta_{\tilde{z}}: G(\afp) \times \xmub(\ovfp) \to \mathscr{I} \subset \shgbinf(\ovfp)$, which is compatible with the product formula map $\pi_{\infty}:\iginf \times \xmub \to \shgbinf$. Therefore the image of $\Theta_{\tilde{z}}$ intersects every central leaf.

To prove that $\Theta_{\tilde{z}}$ surjects onto $\pi_0(\shg)$, we argue as follows: The map $G(\afp) \times \xmub(\ovfp) \to \pi_0(\shginf)$ induced by $\Theta_{\tilde{z}}$ is $G(\afp) \times J_b(\qp)$-equivariant, where $J_b(\qp)$ acts on $\pi_0(\shginf)$ via the surjection $J_b(\qp) \to \pi_1(G)_I^{\sigma}$; this was explained in the proof of Lemma \ref{Lem:JbppActs}. We may then argue as in the proof of Lemma \ref{Lem:JbppActs} to show that $G(\afp) \times \xmub(\ovfp) \to \pi_0(\shginf)$ is surjective.
\end{proof}
\begin{Cor} \label{Cor:MaxTorus}
If Assumption \ref{Assump:1} holds, then for every maximal torus $T \subset J_b$ and every connected component $Z$ of $\iginf$, we can find a maximal torus $T'' \subset J_b'$ such that $T''$ is isomorphic to $T \cap J_b'$ as algebraic groups and such that $T''(\qp)$ stabilises $Z$.
\end{Cor}
\begin{proof}
Let $W$ be the image of $Z$ under $\iginf \to \cbinf$. Then by Proposition \ref{Prop:HondaTate}, we can find a point $z \in W$ such that $I_z$ contains a torus $\mathscr{T}$ with $\mathscr{T}'_{\qp}$ isomorphic to $T \cap J_b'$ and such that $\mathscr{T}'(\mathbb{Q})$ is dense in $\mathscr{T}'(\qp)$. Now let $\tilde{z} \in Z$ be a point lying above $z$, then $\tilde{z}$ induces $j_{\tilde{z},p}:I_{z,\qp} \to J_b$ which sends $I_z'$ to $J_b'$. Thus $T'':=\mathscr{T}'_{\qp}$ is a maximal torus of $J_b'$ which is isomorphic to $T \cap J_b'$ as algebraic groups. Corollary \ref{Cor:ClosureActsTrivially} then tells us that $T''(\qp)$ acts trivially on $Z$.
\end{proof}
\section{Group theory} \label{Sec:GroupTheory}
Let $\mbg$ be a connected reductive group over $\qp$ and assume that $\mbgder$ is simply connected. The goal of this section is to prove Proposition \ref{Prop:GroupTheory} below, which we will apply to $\mbg=J_b'$ in Section \ref{Sec:MainThm}. 

Let $\mbm \subset \mbg$ be the centraliser of a maximal split torus of $\mbg$, and let $\mbmder$ be its derived subgroup. Then $\mbmder$ is simply connected because $\mbgder$ is, see \cite{MalleTesterman}*{Proposition 12.14}, and moreover $\mbmder$ is anisotropic by construction. Therefore, it follows from \cite{PR}*{Theorem 6.5} that there are finite extensions $K_1, \cdots, K_m$ of $\qp$ and central division algebras $D_i$ over $K_i$ such that
\begin{align}
    \mbmder \simeq \prod_{i=1}^m \operatorname{Res}_{K_i/\mathbb{Q}_p} D_i^1.
\end{align}
Here $D_i^1$ is the algebraic group over $K_i$ which is the kernel of the reduced norm map $\operatorname{Nm}_{D_i/K_i}:D_i^{\times} \to \mathbb{G}_{m,K_i}$. We will write $\breve{K}_i$ for the completion of a maximal unramified extension of $K_i$. We will write $\mathcal{P}([m])$ for the set of all subsets $Y \subset \{1,\cdots,m\}=:[m].$
\begin{Prop} \label{Prop:GroupTheory}
For each $Y \in \mathcal{P}([m])$ let $T_Y$ be a maximal torus of $\mbg$ such that: The torus $T_Y$ contains a maximal $\qp$-split torus, has maximal $\breve{K}_i$-split rank among maximal tori containing a maximal $\qp$-split torus for $i \in Y$, and has minimal $\breve{K}_i$-split rank among maximal tori containing a maximal $\qp$-split torus for $i \not \in Y$. Then the group $\mbg(\qp)$ is topologically generated by $\mbgder(\qp)$ and $\bigcup_{Y \in \mathcal{P}([m])} T_Y(\qp)$.
\end{Prop}
Maximal tori as in the statement of Proposition \ref{Prop:GroupTheory} exist, this will be clear from its proof. 
\subsection{Semisimple anisotropic groups}\label{Sec:Anisotropic}
Choose a finite extension $K$ of $\qp$ and a central division algebra $D$ over $K$ of rank $n^2$. Maximal tori of $D^{\times}$ correspond to subfields $F$ of degree $n$ of $D$, and each degree $n$ field extension of $K$ is such a subfield (see \cite{milneCFT}*{Remark IV.4.4.(c)}). For a subfield $F$ of degree $n$ we will write $R^1_{F/K} \mathbb{G}_m \subset \operatorname{Res}_{F/K} \mathbb{G}_m$ for the kernel of the norm map $\operatorname{Nm}_{F/K}:\operatorname{Res}_{F/K} \to \mathbb{G}_{m,K}$, which gives rise to a maximal torus of $D^1$.
\begin{Lem}\label{Lem:Cohomology}
Let $L \subset D$ be a degree $n$ subfield that is unramified over $K$ and let $F \subset D$ be a subfield of degree $n$ that is totally ramified over $K$. Let $Z$ be the center of $D^1$. Then any cohomology class $\alpha \in H^1(K, Z)$ maps to zero in either $H^1(K, R^1_{L/K} \mathbb{G}_m)$ or $H^1(K, R^1_{F/K} \mathbb{G}_m)$.
\end{Lem}
\begin{proof}
We know that $Z \simeq \mu_{n,K}$, and therefore $H^1(K, Z)=K^{\times}/K^{\times,n}$. A standard long exact sequence argument shows that $H^1(K, R^1_{L/K} \mathbb{G}_m)=K^{\times}/\operatorname{Nm}_{L/K} L^{\times}$ and that $H^1(K, R^1_{F/K} \mathbb{G}_m)=K^{\times}/\operatorname{Nm}_{F/K} F^{\times}$. Moreover, under these identifications the natural maps $H^1(K,Z) \to H^1(K, R^1_{L/K} \mathbb{G}_m)$ and $H^1(K,Z) \to H^1(K, R^1_{F/K} \mathbb{G}_m)$ correspond to the natural maps 
\begin{align}
    K^{\times}/K^{\times,n} \to K^{\times}/\operatorname{Nm}_{L/K} L^{\times}, K^{\times}/\operatorname{Nm}_{F/K} F^{\times}
\end{align}
corresponding to the inclusions 
\begin{align}
    K^{\times,n} \subset \operatorname{Nm}_{L/K} L^{\times}, \operatorname{Nm}_{F/K} F^{\times}.
\end{align}
The result follows because the group generated by $\operatorname{Nm}_{L/K} L^{\times}$ and $\operatorname{Nm}_{F/K} F^{\times}$ is equal to $K^{\times}$. Indeed, by local class field theory the group generated by $\operatorname{Nm}_{L/K} L^{\times}$ and $\operatorname{Nm}_{F/K} F^{\times}$ is itself equal to $\operatorname{Nm}_{K'/K} (K')^{\times}$ for a finite abelian extension $K'$ of $K$. Moreover, this extension $K'$ is equal to the intersection of $L$ and $F$ inside a maximal abelian extension $K^{\mathrm{ab}}$ of $K$. Since $L$ is unramified over $K$ and $F$ is totally ramified, this intersection is equal to $K$. Therefore we have an equality $\operatorname{Nm}_{K'/K} (K')^{\times}=\operatorname{Nm}_{K/K} (K)^{\times}=K^{\times}$ and we are done.
\end{proof}
\begin{proof}[Proof of Proposition \ref{Prop:GroupTheory}]
Let $Q \subset \mbg(\qp)$ be the group topologically generated by $\mbgder(\qp)$ and $\bigcup_{Y \in \mathcal{P}([m])} T_Y(\qp)$. Let $K \subset \mbg(\qp)$ be a special parahoric subgroup, then the Cartan decomposition, see \cite[Theorem 1.0.3]{HainesSpecial}, tells us that 
\begin{align} \label{Eq:Cartan}
    \mbg(\qp)=K \mbm(\qp) K,
\end{align}
where $\mbm$ is the centraliser of a maximal split torus $S$ of $\qp$. Thus to show that $Q=\mbg(\qp)$ it suffices to show that $Q$ contains $\mbm(\qp)$ and that $Q$ contains $K$. Note that $\mbmder$ is simply connected since $\mbm$ is a Levi subgroup of $\mbg$, see \cite{MalleTesterman}*{Proposition 12.14}.

\textbf{Step 1: The group $Q$ contains $\mbm(\qp)$.} By the vanishing of Galois cohomology for semi-simple simply connected groups over $p$-adic local fields, there is a short exact sequence
\begin{align}
    1 \to \mbmder(\qp) \to \mbm(\qp) \to \mbmab(\qp) \to 1.
\end{align}
Moreover since $\mbmder$ is semisimple, we see that $\mbmder \subset \mbgder$ and thus $\mbmder(\qp) \subset Q$. 

For each $Y \in \mathcal{P}([m])$, the torus $T_Y$ contains the maximal split torus of $\mbg$ by assumptions, which is $\mbgder(\qp)$-conjugate to $S$. Since $T_Y(\qp) \subset Q$ and since $Q$ contains $\mbgder(\qp)$, we may assume without loss of generality that $T_Y$ contains $S$ and is thus contained in $\mbm$ for all $Y$. To complete the proof of step 1, it thus suffices to show that the group (topologically) generated by $\bigcup_{Y \in \mathcal{P}([m])} T_Y(\qp)$ surjects onto $\mbmab(\qp)$. \smallskip 

The short exact sequences (for $Y \in \mathcal{P}([m])$)
\begin{equation}
    \begin{tikzcd}
    1 \arrow{r} & T_Y \cap \mbmder \arrow{r} & T_Y \arrow{r} & \mbmab \arrow{r} & 1,
    \end{tikzcd}
\end{equation}
induce long exact sequences (for $Y \in \mathcal{P}([m])$)
\begin{equation}
    \begin{tikzcd}
    1 \arrow{r} & (T_Y \cap \mbmder)(\qp) \arrow{r} & T_Y(\qp) \arrow{r} & \mbmab(\qp) \arrow{r} & H^1(\qp, T_Y \cap \mbmder) \arrow{r} & \cdots
    \end{tikzcd}
\end{equation}
We deduce that it is enough to prove that every element of $\mbmab(\qp)$ maps to zero in $H^1(\qp, T_Y \cap \mbmder)$ for some $Y$. Recall the notation from the beginning of Section \ref{Sec:Anisotropic}. By our assumptions on the tori $T_Y$, we find that
\begin{align}
    T_Y \cap \mbmder = \prod_{i \in Y} \operatorname{Res}_{K_i/\mathbb{Q}_p} R^1_{L_i/K_i} \mathbb{G}_m \times \prod_{i \not \in Y} \operatorname{Res}_{K_i/\mathbb{Q}_p} R^1_{F_i/K_i} \mathbb{G}_m, 
\end{align}
where $L_i \subset D_i$ is a maximal subfield that is unramified over $K_i$, and $F_i \subset D_i$ is a maximal subfield that is totally ramified over $K_i$. Now we note that the natural maps (for $Y \in \mathcal{P}([m])$)
\begin{align}
    \mbmab(\qp) \to H^1(\qp, T_Y \cap \mbmder)
\end{align}
factor through $\mbmab(\qp) \to H^1(\qp, Z_{\mbmder})$, where $Z_{\mbmder}$ denotes the center of $\mbmder$. But by Lemma \ref{Lem:Cohomology} in combination with Shapiro's lemma, every element of $H^1(\qp, Z_{\mbmder})$ maps to zero in $H^1(\qp, T_Y \cap \mbmder)$ for some $Y \in \mathcal{P}([m])$. We conclude that $Q$ surjects onto $\mbmab(\qp)$ and therefore contains $\mbm(\qp)$.

\textbf{Step 2: The group $Q$ contains a special parahoric subgroup.} For $Y=[m]$, by the description of $T_Y \cap \mbmder$ above, we see that it has maximal $\qpbr$-split rank. It follows that $T_Y \cap \mbm$ has maximal $\qpbr$-split rank, which means that $T_Y$ has maximal $\qpbr$-split rank among maximal tori that have maximal $\qp$-split rank; we deduce that $T_Y$ has maximal $\qpbr$-split rank among \emph{all} maximal tori of $\mbg$. This moreover means that $T_Y \otimes \qpbr$ is a maximal $\qpbr$-split torus. Because $\mbgder$ is simply connected, we see that $X_{\ast}(T_Y\cap \mbgder)$ is an induced Galois module for the action of the inertia group $I$, see \cite{BTII}*{Proposition 4.4.16}, and so $X_{\ast}(T_Y \cap \mbgder)_I$ is torsion free.

Let $\mathcal{T}_Y$ be the connected N\'eron-model of $T_Y$ and let $\mathcal{T}_Y^{\mathrm{der}}$ be the connected N\'eron-model of $T_Y \cap \mbgder$. Then since $X_{\ast}(T_Y \cap \mbgder)_I$ is torsion free, it follows from \cite[Lemma 6.7]{PappasRapoportAffineFlag} that there is a short exact sequence
\begin{align}
    1 \to \mathcal{T}_Y^{\mathrm{der}} \to \mathcal{T}_Y \to \mathcal{D} \to 1,
\end{align}
where $\mathcal{D}$ is the connected N\'eron model of $\mbgab$. Since $\mathcal{T}_Y^{\mathrm{der}}$ has connected special fibre, it follows from Lang's lemma that $\mathcal{T}_Y(\zp) \subset T_Y(\qp)$ surjects onto $\mathcal{D}(\zp)$. Thus the image of $Q$ in $\mbgab(\qp)$ contains $\mathcal{D}(\zp)$, and since $\mbgder(\qp) \subset Q$ it follows that $Q$ contains the inverse image of $\mathcal{D}(\zp)$. Parahoric subgroups of $\mbg(\qp)$ map to $\mathcal{D}(\zp)$ by Proposition \ref{Prop:ImageParahoricII} and hence $Q$ contains every parahoric subgroup of $\mbg(\qp)$. 
\end{proof}
\section{Main theorems} \label{Sec:MainThm}
In this section we will state the main theorems in full generality, giving Theorems \ref{Thm:Monodromy} and \ref{Thm:Newton} as special cases. We first recall our running assumptions and some notation. 

\subsubsection{} Let $(G,X)$ be a Shimura datum of Hodge type with reflex field $E$ and assume that $\gder$ is simply connected. Let $p>2$ be a prime number such that $G_{\qp}$ is quasi-split and splits over a tamely ramified extension. Let $K^p \subset G(\afp)$ be a sufficiently small compact open subgroup and let $K_p \subset G(\qp)$ be a connected very special parahoric subgroup, where ``connected'' is used in the sense of \cite[start of Section 2]{Zhou}. (Note that hyperspecial parahoric subgroups are automatically connected, see \cite[Remark 4.2.14.b)]{KisinPappas}. In particular, we don't have to worry about this assumption when deducing the main theorems of the introduction.) Choose a Hodge embedding $(G,X) \to (\mathcal{G}_V, \mathcal{H}_V)$ and $\mathbb{Z}_{(p)}$-lattice $V_{(p)} \subset V$ on which $\psi$ is $\mathbb{Z}_{(p)}$-valued, such that $K_p$ is the stabiliser of $V_p$ in $G(\qp)$; this is always possible by the discussion in \cite[Section 1.3.2]{KMPS}. Let $v |p $ be a prime of $E$ and let $\shg$ be the geometric special fiber of the integral model of the Shimura variety of level $K^pK_p$, see Section \ref{Sec:Models}.

\subsubsection{} Let $[b] \in B(G,\{\mu^{-1}\})$ be a $\mathbb{Q}$-non-basic element as defined in Section \ref{Sec:Monodromy}. Let $x \in \shgb(\ovfp)$ be an element contained in a distinguished central leaf (such $x$ exist by Lemma \ref{Lem:ExistenceDistinguished}) and let $b=b_x$ be as in the first paragraph of Section \ref{Sec:Monodromy}. Let $\ig \to \cb$ be the Igusa variety associated to $x$ as constructed in Section \ref{Sec:Igusa}. For $y \in \xmub(\ovfp)$, we will also consider the Igusa variety as a pro-\'etale $H_y$-torsor over $\cby^{\mathrm{perf}}$ using Corollary \ref{Cor:IgusaVarietiesAsTorsors}. In this notation, the Igusa variety $\ig \to \cb$ corresponds to $y=1 \in \xmub(\ovfp)$, and we will write $H=H_{1} \subset J_b(\qp)$ for its stabiliser. 

We write $J_b$ for the $\sigma$-centraliser of $b$. Note that $\jbder$ is simply connected since $J_b$ is an inner form of a Levi subgroup of $G_{\qp}$, see \cite{MalleTesterman}*{Proposition 12.14}. 
\begin{Assump} \label{Assump:2}
The group $J_b^{\mathrm{der}}$ has no compact factors.
\end{Assump}

This assumption means that if we write $J_b^{\mathrm{der}}$ as a product of $\qp$-almost-simple groups $\mbg_1 \times \cdots \times \mbg_n$, then $\mbg_i(\qp)$ is not compact in the $p$-adic topology for any $i$. Since $\jbder$ is simply connected, this is equivalent to asking the same for $J_b^{\mathrm{ad}}$.

Recall the subgroup $J_b''(\qp) \supset J_b'(\qp)$, which is the inverse image of $\gab(\zp) \subset \gab(\qp)$ under $J_b(\qp) \to \gab(\qp)$. Since $H_y$ is contained in a parahoric subgroup of $J_b(\qp)$, for example by the proof of Lemma \ref{Lem:IsParahoric}, it is contained in $J_b''(\qp)$ by Proposition \ref{Prop:ImageParahoricI}. We can now state our main theorem.
\begin{Thm} \label{Thm:Monodromy2}
Let $(G,X)$ and $[b]$ be as above. If Assumptions \ref{Assump:1} and \ref{Assump:2} hold, then for $y \in \xmub(\ovfp)$ the natural map
\begin{align}
    \pi_0(\ig \times \{y\}) \to \pi_0(\shg) 
\end{align}
is surjective with fibers in bijection with $G^{\mathrm{ab}}(\zp)$, equivariant for the action of $J_b''(\qp)$ (which stabilises the fibers by Lemma \ref{Lem:JbppActs}). In particular, the identification is $H_y$-equivariant for the natural action of $H_y$ on the fibers.
\end{Thm}
Assumption \ref{Assump:1} holds true if $G_{\qp}$ splits over an unramified extension by \cite{vH}*{Theorem 4.5.2}; the assumption in \cite{vH}*{Theorem 4.5.2} that \cite[Conjecture 4.3.1]{vH} holds is satisfied when $G_{\qp}$ is unramified, see \cite[Remark 4.3.2]{vH}. In particular, Theorem \ref{Thm:Monodromy2} implies Theorem \ref{Thm:Monodromy}. Moreover, Assumption \ref{Assump:1} also holds when $\mathbf{Sh}_{K^pK_p}(G,X)$ is proper, by \cite[Theorem 4.5.2]{vH}.
\begin{Rem} \label{Rem:MuOrdinary}
When $[b]$ is the $\mu$-ordinary element (see Section \ref{Sec:NewtonStratification} for the definition), then $J_b$ is quasi-split, which implies that Assumption \ref{Assump:2} holds. Moreover, Assumption \ref{Assump:1} holds because in this case $\cb=\shgb$ and $\shgb \subset \shg$ is dense by \cite[Theorem 3]{KMPS}. Indeed, the assumption that $\shg$ is locally integral in the statement of \cite[Theorem 3]{KMPS} holds because $K_p$ is very special, see \cite[Corollary 4.6.26]{KisinPappas}.
\end{Rem}
\begin{proof}[Proof of Theorem \ref{Thm:Monodromy2}]
\textbf{Step 1: The group $J_b'(\qp)$ acts trivially on $V$.} Fix a connected component $V$ of $\ig$. By Proposition \ref{Prop:Geometricmonodromy} and Assumption \ref{Assump:1}, the group $\jbderiso(\qp)$ acts trivially on $\pi_0(\ig)$. By Assumption \ref{Assump:2} we have $\jbderiso=\jbder$ and so $\jbder(\qp)$ acts trivially on $\pi_0(\ig)$. By Corollary \ref{Cor:MaxTorus} we can find, for every isomorphism class of maximal tori of $J_b'$, a representative $T' \subset J_b'$ such that $T'(\qp)$ stabilises $V$. 

Recall that $J_b'$ is connected reductive by Lemma \ref{Lem:JbpConnected}. It thus follows from Proposition \ref{Prop:GroupTheory} that we can find maximal tori $T_1, \cdots, T_n$ of $J_b'$, which can be specified up to isomorphism, such that the group topologically generated by $T_1(\qp), \cdots, T_n(\qp)$ and $\jbder(\qp)$ is equal to $J_b'(\qp)$. Since the stabiliser of $V$ in $J_b(\qp)$ is closed, see the proof of Lemma \ref{Cor:ClosureActsTrivially}, it follows that $J_b'(\qp)$ acts trivially on $V$. Since $V$ was chosen arbitrarily, this implies that $J_b'(\qp)$ acts trivially on $\pi_0(\ig)$. \smallskip

\textbf{Step 2: The theorem for $y=1$.} Assumption \ref{Assump:1} tells us that
\begin{align}
    \pi_0(\cb) \to \pi_0(\shg)
\end{align}
is a bijection, and therefore the fibers of $\pi_0(\ig) \to \pi_0(\shg)$ are in bijection with the fibers of $\pi_0(\ig) \to \pi_0(\cb)$. Now $\ig \to \cb^{\mathrm{perf}}$ is an $H$-torsor and the image of monodromy is contained in $H'=J_b'(\qp) \cap H$ by Proposition \ref{Prop:Geometricmonodromy}. Recall that the natural map $\cb^{\mathrm{perf}} \to \cb$ is a homeomorphism. The fact that $H'$ acts trivially on $\pi_0(\ig)$ then implies that the fibers of $\pi_0(\ig) \to \pi_0(\cb)$ are in bijection with $H/H'$. Since $\cb$ is a distinguished central leaf, it follows from Lemma \ref{Lem:IsParahoric} that there are parahoric subgroups $J,J'$ of $J_b(\qp)$ such that $J \subset H \subset J'$. In particular, it follows from Proposition \ref{Prop:ImageParahoricI} that $H/H' \simeq \gab(\zp)$. 

Now $J_b''(\qp)$ stabilises the fibers of $\pi_0(\ig) \to \pi_0(\shg)$ by Lemma \ref{Lem:JbppActs} and since $J_b'(\qp) \subset J_b''(\qp)$ acts trivially on $\pi_0(\ig)$ this action factors through an action of $J_b''(\qp)/J_b(\qp)=\gab(\zp)$. Since $H$ surjects onto $\gab(\zp)$, we see that our identification of the fibers with $\gab(\zp)$ is also $J_b''(\qp)$-equivariant. 

\textbf{Step 3: The theorem for arbitrary y.} For $y \in \xmub(\ovfp)$ we consider the morphism
\begin{align}
    \pi_{\infty}(-,y):\pi_0(\ig \times \{y\}) \to \pi_0(\cby) \to \pi_0(\shg).
\end{align}
We note that this morphism only depends on the connected component of $\xmub$ containing $\{y\}$. For $y \in \xmub(\ovfp)$ lying in the same component as $1 \in \xmub(\ovfp)$, the theorem therefore follows from the discussion above. Since $J_b(\qp)$ acts transitively on $\pi_0(\xmub)$ by \cite[Theorem A.1.3]{vH}, it suffices to prove the result for $y \in \operatorname{Orb}(1)$, where $\operatorname{Orb}(1) \subset \xmub(\ovfp)$ is the $J_b(\qp)$-orbit of $1$. 

But for $j \in J_b(\qp)$ we have $\pi_{\infty}(z, j \cdot 1) = \pi_{\infty}(j^{-1} z, 1)$ and so the fibers of $\pi_{\infty}(-,j \cdot 1)$ can be identified with the fibers of $\pi_{\infty}(j^{-1} z, 1)$ under the isomorphism $j:\ig \to \ig$. This identification of the fibers is $J_b''(\qp)$-equivariant for the precomposition of the natural $J_b''(\qp)$-action on the fibers of $\pi_{\infty}(-,j \cdot 1)$ with conjugation by $j$ (considered as an automorphism of $J_b''(\qp)$).

Since the action of $J_b''(\qp)$ on the fibers of $\pi_{\infty}(-, 1)$ identifies these fibers with principal homogeneous spaces for $\gab(\zp)$, the same is true for the $j$-twisted action of $J_b''(\qp)$ on the fibers of $\pi_{\infty}(-, j \cdot 1)$. Since $\gab(\zp)$ is abelian, it follows that the untwisted $J_b''(\qp)$-action on the fibers of $\pi_{\infty}(-, j \cdot 1)$ also identifies these fibers with principal homogeneous spaces for $\gab(\zp)$; the theorem is proved. 
\end{proof}
We now state the general version of Corollary \ref{Cor:LeavesIrred}.
\begin{Cor} \label{Cor:LeavesIrred2}
Let $(G,X)$ and $[b]$ be as above. If Assumptions \ref{Assump:1} and \ref{Assump:2} hold, then for $y \in \xmub(\ovfp)$, the natural map
\begin{align}
    \pi_0(\cy) \to \pi_0(\shg)
\end{align}
is surjective with finite fibers. Moreover, the fibers are in bijection with $H_y \backslash \gab(\zp)$.
\end{Cor}
\begin{proof}
Let $\pi_0(\ig) \to \pi_0(\cy)$ be the map induced from $\ig \times \{y\} \to \cy^{\mathrm{perf}}$. Then the fibers of this map are a subset of the fibers of the composition $\pi_0(\ig) \to \pi_0(\shg)$ of our map with $\pi_0(\cy) \to \pi_0(\shg)$. To determine this subset, we observe that $$\ig \times \{y\} \to \cy^{\mathrm{perf}}$$ is a $H_y$-torsor by Corollary \ref{Cor:IgusaVarietiesAsTorsors}. Moreover, $H_y \subset J_b''(\qp)$ acts on the fibers of $\pi_0(\ig) \to \pi_0(\shg)$ with stabiliser $H_y':=H_y \cap J_b'(\qp)$, by the $J_b''(\qp)$-equivariance of Theorem \ref{Thm:Monodromy2}. 

Therefore the fibers of $\pi_0(\ig) \to \pi_0(\cy)$ can be identified with $H_y/H_y' \subset \gab(\zp)$ and the fibers of $\pi_0(\cy) \to \pi_0(\shg)$ can be identified with $H_y \backslash \gab(\zp)$.
\end{proof}
Now we state the generalisation of Theorem \ref{Thm:Newton}. If $K_p$ is hyperspecial, then the representation-theoretic constant $\operatorname{Dim} V^{\hat{H}}_{\underline{\mu}}(\lambda_b)_{\text{rel}}$ is equal to $\operatorname{Dim} V_{\mu}(\lambda_b)_{\mathrm{rel}}$ from the statement of Theorem \ref{Thm:Newton}. If $G_{\qp}$ is moreover split, then $\operatorname{Dim} V_{\mu}(\lambda_b)_{\mathrm{rel}}=1$; this is a straightforward consequence of the definition of $V_{\mu}(\lambda_b)_{\mathrm{rel}}$ in \cite[Section 2.6]{ZhouZhu} and the fact that $\{\mu\}$ is minuscule. In particular, Theorem \ref{Thm:Newton2} implies Theorem \ref{Thm:Newton}.
\begin{Thm} \label{Thm:Newton2}
Let $(G,X)$ and $[b]$ be as above. If Assumptions \ref{Assump:1} and \ref{Assump:2} hold, then for $y \in \xmub(\ovfp)$ the natural map 
\begin{align}
    \pi_0(\shgb) \to \pi_0(\shg)
\end{align}
is a bijection. Moreover, the number of irreducible components in each connected component of $\shgb$ is given by the representation-theoretic constant
\begin{align}
    \operatorname{Dim} V^{\hat{H}}_{\underline{\mu}}(\lambda_b)_{\text{rel}},
\end{align}
introduced in \cite[Section A.3]{ZhouZhu}. 
\end{Thm}
\begin{proof}
By Proposition \ref{Prop:ProductFormulaLemma} there is a commutative diagram
\begin{equation}
    \begin{tikzcd}
    \ig \times \operatorname{Orb}(1) \arrow{d} \arrow{r} & \ig \times \xmub \arrow{d} \\
    \cb^{\mathrm{perf}} \arrow{r} & \shgb^{\mathrm{perf}} \arrow{r} & \shg^{\mathrm{perf}}.
    \end{tikzcd}
\end{equation}
Assumption \ref{Assump:1} tells us that the composite map $\cb \to \shg$ induces a bijection 
\begin{align}
    \pi_0(\cb) \to \pi_0(\shg).
\end{align}
To prove the first part of the theorem, it thus suffices to prove that $\pi_0(\cb) \to \pi_0(\shgb)$ is a bijection, and since the injectivity follows from the injectivity of $\pi_0(\cb) \to \pi_0(\shg)$, it is enough to show that $\pi_0(\cb) \to \pi_0(\shgb)$ is surjective. Using the $J_b(\qp)$-torsor structure of $\ig \times \operatorname{Orb}(1) \to \cb^{\mathrm{perf}}$ and $\ig \times \xmub \to \shgb^{\mathrm{perf}}$, see Section \ref{Sec:Product} and Proposition \ref{Prop:ProductFormulaLemma}, this comes down to showing that the map
\begin{align}
    \ig \times \operatorname{Orb}(1) \to \ig \times \xmub
\end{align}
induces a surjection on connected components. By \cite[Theorem A.1.3]{vH}, the group $J_b(\qp)$ acts transitively on $\pi_0(\xmub)$, thus $\operatorname{Orb}(1) \to \xmub$ induces a surjection on $\pi_0$ which implies that $\ig \times \operatorname{Orb} y \to \ig \times \xmub$ induces a surjection on $\pi_0$. We deduce that the natural maps $\pi_0(\cb) \to \pi_0(\shgb) \to \pi_0(\shg)$ are all bijections. If we moreover let $\operatorname{Stab}_X \subset J_b(\qp)$ be the stabiliser of a connected component $X$ of $\xmub$, then $\operatorname{Stab}_X$ acts on $\pi_0(\ig)$ and it follows from the above reasoning that the natural map 
\begin{align} \label{Eq:FactComponents}
    \operatorname{Stab}_X \backslash \pi_0(\ig) \to \pi_0(\shgb)
\end{align}
is a bijection. \smallskip

For the second part of the theorem, we need to compute the set of $J_b(\qp)$-orbits of irreducible components of $\ig \times \xmub$. By \cite[Theorem A.3.1]{ZhouZhu}, the number of $J_b(\qp)$-orbits of irreducible components in $\xmub$ is given by 
\begin{align}
    N:=\operatorname{Dim} V^{\hat{H}}_{\underline{\mu}}(\lambda_b)_{\text{rel}}.
\end{align}
Let us choose representatives $a_1, \cdots, a_N$ of these orbits, with stabilisers $\operatorname{Stab}_{a_1}, \cdots, \operatorname{Stab}_{a_N}$ in $J_b(\qp)$. Let $\Sigma(\shgb)$ denote the set of irreducible components of $\shgb$, and also for other schemes. By the product formula, see Section \ref{Sec:Product}, the map 
\begin{align}
    \pi_{\infty}:\ig \times \xmub \to \shgb^{\mathrm{perf}}
\end{align}
is a $J_b(\qp)$-torsor. Thus we can write
\begin{align}
    \Sigma(\shgb) &= J_b(\qp) \backslash \left( \Sigma(\ig) \times \Sigma(\xmub) \right) \\
    &=J_b(\qp) \backslash \left( \pi_0(\ig) \times \Sigma(\xmub) \right) \\
    &=J_b(\qp) \backslash \left( \pi_0(\ig) \times \coprod_{i=1}^N \operatorname{Stab}_{a_i} \backslash J_b(\qp) \right) \\
    &= \coprod_{i=1}^N \operatorname{Stab}_{a_i} \backslash \pi_0(\ig).
\end{align}
For each $i$ we let $X_{a_i}$ be the connected component of $\xmub$ containing $a_i$. Then $\operatorname{Stab}_{a_i} \subset \operatorname{Stab}_{X_{a_i}}$ and moreover the map
\begin{align}
    \Sigma(\shgb) \to \pi_0(\shgb)
\end{align}
can be identified with the map
\begin{align}
    \coprod_{i=1}^N \operatorname{Stab}_{a_i} \backslash \pi_0(\ig) \to \pi_0(\shgb)
\end{align}
induced by the maps (the second map comes from equation \eqref{Eq:FactComponents})
\begin{align} \label{Eq:IgModStabToIgModStab}
    \operatorname{Stab}_{a_i} \backslash \pi_0(\ig) \to \operatorname{Stab}_{X_{a_i}} \backslash \pi_0(\ig) = \pi_0(\shgb). 
\end{align}
In particular, we see that it suffices to prove that the map in \eqref{Eq:IgModStabToIgModStab} is a bijection for all $i$. By Assumptions \ref{Assump:1} and \ref{Assump:2} we may invoke Theorem \ref{Thm:Monodromy2}, and we see that it suffices to show that $\operatorname{Stab}_{a_i} \subset J_b(\qp)$ surjects onto $\gab(\zp)$ for all $i$. But $\operatorname{Stab}_{a_i}$ is a parahoric subgroup by \cite[Theorem 3.1.1]{ZhouZhu}, and therefore it surjects onto $\gab(\zp)$ by Proposition \ref{Prop:ImageParahoricI}.

\end{proof}
\subsection{A conjectural description of the connected components of Igusa varieties} By Corollary \ref{Cor:Transitive}, the group $G(\afp) \times J_b(\qp)$ acts transitively on $\pi_0(\iginf)$. Moreover, under Assumptions \ref{Assump:1} and \ref{Assump:2} the group $J_b'(\qp)$ acts trivially by the proof of Theorem \ref{Thm:Monodromy2}, and the group $\gder(\afs)$ acts trivially by Lemma \ref{Lem:PrimeToPDerivedTrivial}. We have the following conjectural description of $\pi_0(\iginf)$.
\begin{Conj} \label{Conj:Igusa}
There is a $G(\afp) \times J_b(\qp)$-equivariant isomorphism of topological spaces
\begin{align}
    \pi_0(\iginf) = \gab(\mathbb{Q})^{\dag} \setminus \gab(\af),
\end{align}
where $J_b(\qp)$ acts via $J_b(\qp) \to \gab(\qp)$ and where $\gab(\mathbb{Q})^{\dag}$ is as in Section \ref{Sec:ProductFormulaAndPiNaught}.
\end{Conj}
\begin{Rem}
Assuming the conjecture, we get an automorphic description of $H^0_{\text{\'et}}(\iginf, \qlbar)$ as in \cite[Theorem A]{KretShin}. Indeed, there is a $\gab(\af)$-equivariant bijection
\begin{align}
    \gab(\mathbb{Q})^{\dag} \setminus \gab(\af) = \varprojlim_{K} \pi_0(\mathbf{Sh}_{K,\mathbb{C}}(G,X)),
\end{align}
and the zeroth \'etale cohomology of the right-hand side \emph{has} an automorphic description, as discussed in \cite[Section 5.1]{KretShin}.
\end{Rem}
Suppose that the conclusion of Theorem \ref{Thm:Monodromy2} holds. Then Conjecture \ref{Conj:Igusa} would follow if the following question had a positive answer.
\begin{Question} \label{Q:ToriGeneratedGab}
Is it true that the images of $I_z(\mathbb{Q})$ in $\gab(\mathbb{Q})^{\dag}$, as $\tilde{z}$ ranges over all the points in $\iginf(\ovfp)$, generate $\gab(\mathbb{Q})^{\dag}$?
\end{Question}
For the Igusa variety over the ordinary locus in the modular curve, this question asks if every $q \in \mathbb{Q}_{>0}^{\times}=\operatorname{GL}_2^{\mathrm{ab}}(\mathbb{Q})^{\dag}$ is equal to the norm of an element in an imaginary quadratic field $E$ where our fixed prime $p$ splits; the answer to this question is yes. 
\subsection{The discrete Hecke orbit conjecture} 
For the benefit of the reader, we recall the statement of the discrete Hecke orbit conjecture and its stronger version from \cite{KretShin}. Let the notation be as in Section \ref{Sec:Notation}, then the following two conjectures are \cite[Question 8.2.1, last two bullet points]{KretShin}.
\begin{Conj}[The strong discrete Hecke orbit conjecture] \label{Conj:DiscreteHOStrong}
The natural map $\pi_0(\cby) \to \pi_0(\shg)$ is a bijection for all $y$ in $\xmub(\ovfp)$ and all $K^p \subset G(\afp)$.
\end{Conj}
\begin{Conj}[The discrete Hecke orbit conjecture] \label{Conj:DiscreteHOWeak}
If $G_{\qp}$ is unramified, then $G(\afp)$ acts transitively on $\pi_0(\cbyinf)$ for all $y$ in $\xmub(\ovfp)$.
\end{Conj}
If $G_{\qp}$ is not ramified, then $G(\afp)$ does not necessarily act transitively on $\pi_0(\shginf)$, see \cite{OkiComponents} for explicit counterexamples. Therefore the assumption that $G_{\qp}$ is unramified is necessary in Conjecture \ref{Conj:DiscreteHOWeak}. Conjecture \ref{Conj:DiscreteHOWeak} follows from Conjecture \ref{Conj:Igusa} using weak approximation for the torus $\gab$, which holds since $\gab_{\qp}$ splits over an unramified extension by the assumptions of \ref{Conj:DiscreteHOWeak}. Note that Conjecture \ref{Conj:DiscreteHOWeak} is proved by Kret--Shin as \cite[Theorem 8.2.6]{KretShin}. \smallskip

Conjecture \ref{Conj:DiscreteHOStrong} implies Conjecture \ref{Conj:DiscreteHOWeak} because $G(\afp)$ acts transitively on $\pi_0(\shginf)$ when $G_{\qp}$ is unramified, see \cite[Lemma 2.2.5]{KisinModels} and \cite[Corollary 4.1.11]{MP}. 

\subsection{A counterexample to the strong version of the discrete Hecke orbit conjecture}\label{counterexample}
The purpose of this section is to show that Conjecture \ref{Conj:DiscreteHOStrong} is false; we will present below a counterexample communicated to us by Rong Zhou. As noted in Remark \ref{Rem:ChaiOortLeaves}, it would be interesting to find a counterexample with $(G,X)=(\mathcal{G}_V, \mathcal{H}_V)$.

Our counterexample involves a unitary Shimura variety of PEL type. Let $F=F^+E$ be a CM field where $F^+$ is totally real of degree $4$ and $E$ is an imaginary quadratic field. Let $V$ be a Hermitian $F$-vector space of rank $2$ with signature $(1,1)$ at all infinite places of $F^+$ and let $(G,X)=(\operatorname{GU}_V, X_V)$ be the corresponding Shimura datum of PEL type. 

Let $p>2$ be a prime which splits in $E$ and which splits as $p=\mathfrak{p}_1 \mathfrak{p}_2$ in $F^+$, such that both $\mathfrak{p}_1$ and $\mathfrak{p}_2$ have residual degree $2$. Then
\begin{align}
    G_{\qp} \simeq \operatorname{Res}_{K/\qp} \operatorname{GL}_2 \times \operatorname{Res}_{K/\qp} \operatorname{GL}_2 \times \mathbb{G}_m,
\end{align}
where $K$ is the unique unramified quadratic extension of $\qp$. Let $[b] \in B(G,\{\mu^{-1}\})$ be the unique element which is $\mu_1$-ordinary in the first factor and which is basic in the second factor and third factor (see the third paragraph of Section \ref{Sec:NewtonStratification} for a definition of these terms). Here $\{\mu\}=(\{\mu_1\},\{\mu_2\},\{\mu_3\})$ according to the product decomposition of $G_{\qp}$. Let $b \in [b]$ and write $b=(b_1, b_2, b_3)$ using the product decomposition of $G_{\qp}$. There is an induced product decomposition
\begin{align}
    \xmub= \prod_{i=1}^3 X_{\{\nu_i\}}(b_i),
\end{align}
where $\{\nu_i\}=\{\sigma(\mu_i^{-1})\}$. Thus an element $y \in \xmub(\ovfp)$ has the form $y=(y_1,y_2,y_3)$ and so its stabiliser $H_y \subset J_b(\qp)=J_{b_1}(\qp) \times J_{b_2}(\qp) \times J_{b_3}(\qp)$ can be written as a product
\begin{align}
    H_y=H_{y_1} \times H_{y_2} \times H_{y_3}.
\end{align}
By Corollary \ref{Cor:LeavesIrred2}, if we can find some $y_2 \in X_{\{\nu_2\}}(b_2)$ such that $H_{y_2}$ does not surject onto the $\zp$-points of the maximal abelian quotient of $\operatorname{Res}_{K/\qp} \operatorname{GL}_2$, then $\pi_0(\cby) \to \pi_0(\shg)$ is not a bijection and thus Conjecture \ref{Conj:DiscreteHOStrong} will be false. \smallskip

In this case, there is an isomorphism $J_{b_2} \simeq \operatorname{Res}_{K/\qp} \operatorname{GL}_2$ and $X_{\{\nu_2\}}(b_2)$ is equidimensional of dimension $1$.
\begin{Lem}
The irreducible components of $X_{\{\nu_2\}}(b_2)$ are isomorphic to $\mathbb{P}^1$. The stabiliser of an irreducible component is a hyperspecial subgroup of $\operatorname{GL}_2(K)$, conjugate to $\operatorname{GL}_2(\mathcal{O}_K)$, which acts on $\mathbb{P}^1$ via the natural map $\operatorname{GL}_2(\mathcal{O}_K) \to \operatorname{GL}_2(\mathbb{F}_{p^2})$.
\end{Lem}
\begin{proof}
It follows from the main result of \cites{GortzHeNie} that $X_{\{\nu_2\}}(b_2)$ is a union of two one-dimensional Ekedahl--Oort strata and one zero-dimensional Ekedahl--Oort stratum. The irreducible components are the closures of the one-dimensional Ekedahl--Oort strata. It follows from \cite[Section 5.10]{GortzHeNie} that the irreducible components are unions of (closures of) classical Deligne--Lusztig varieties for the group $\operatorname{Res}_{\mathbb{F}_{p^2}/\mathbb{F}_p} \operatorname{GL}_2$. Therefore, their irreducible components are isomorphic to $\mathbb{P}^1$ and the action of $\operatorname{GL}_2(\mathcal{O}_K)$ factors through the natural action of $\operatorname{GL}_2(\mathbb{F}_{p^2})$.
\end{proof}
A direct computation shows that we can find a point $a\in \mathbb{P}^1(\ovfp)$ such that its stabiliser in $\operatorname{GL}_2(\mathbb{F}_{p^2})$ is given by the group of scalar matrices. Note that this stabiliser does not surject onto $\mathbb{F}_{p^2}^{\times}$ via the determinant map. This implies that the stabiliser of $a$ in $\operatorname{GL}_2(\mathcal{O}_K)$, which is the stabiliser of $a \in X_{\{\nu_2\}}(b_2)(\ovfp)$ in $J_b(\qp)$, does not surject onto $\mathcal{O}_K^{\times}$ via the determinant map. 
\subsection{Acknowledgements} 
We would like to thank Marco D'Addezio for answering our MathOverflow question, making his work in progress \cites{d2020monodromy} available to us and for discussions about Proposition \ref{Prop:FactIsocrystals}. We would also like to thank Brian Conrad, Sean Cotner, Andrew Graham, Xuhua He, João Lourenço, Ambrus Pal, Syed Waqar Ali Shah, Sug Woo Shin, Jize Yu and Rong Zhou for helpful discussions and correspondences. We are grateful to James Newton for comments on an earlier version of this work. We would like to thank the referee for a careful reading of a first version of our paper, and for many helpful suggestions on a first and second version.

\DeclareRobustCommand{\VAN}[3]{#3}

%\bibliographystyle{amsalpha}
%\bibliography{references.bib}
\printbibliography
\end{document}